\documentclass[reqno, 11pt]{amsart}

\usepackage{enumitem}
\setenumerate[0]{label=(\roman*)}

\usepackage[T1]{fontenc}    
\usepackage[a4paper, hmargin={2.7cm,2.7cm},vmargin={3.3cm,3.3cm}]{geometry}
\usepackage{amssymb}   
\usepackage{amsthm}    
\usepackage{thmtools}  
\usepackage{mathtools} 
\usepackage{mathrsfs}  
\usepackage{commath}
\usepackage{bbm}
\usepackage[textwidth=20mm]{todonotes}

\RequirePackage[backend    = biber,
                sortcites  = true,
                giveninits = true,
                doi        = false,
                isbn       = false,
                url        = false,
                maxnames = 50,
                style      = numeric-comp]{biblatex}
\DeclareNameAlias{sortname}{family-given}
\DeclareNameAlias{default}{family-given}
\usepackage{varioref}
\usepackage{hyperref}
\usepackage[nameinlink, capitalize, noabbrev]{cleveref}

\declaretheorem[style = plain, numberwithin = section]{theorem}
\declaretheorem[style = plain,      sibling = theorem]{corollary}
\declaretheorem[style = plain,      sibling = theorem]{lemma}
\declaretheorem[style = plain,      sibling = theorem]{proposition}
\declaretheorem[style = definition, sibling = theorem]{definition}

\declaretheorem[style = definition, sibling = theorem]{observation}
\declaretheorem[style = definition, sibling = theorem]{example}

\declaretheorem[style = definition, sibling = theorem]{remark}

\declaretheorem[style = plain,      sibling = theorem]{assumption}

\newcommand{\N}{\mathbb{N}}
\newcommand{\Z}{\mathbb{Z}}

\newcommand{\R}{\mathbb{R}}
\newcommand{\C}{\mathbb{C}}
\newcommand{\T}{\mathbb{T}}

\DeclareMathOperator{\spn}{span}
\DeclareMathOperator{\clspn}{\overline{\spn}}

\DeclareMathOperator{\covol}{covol}

\DeclareMathOperator{\Ell}{L}
\def\lregw#1{\Ell_{#1}}
\def\lreg#1#2{\lregw{#1}(#2)}
\def\lregpoint#1#2{\Ell_{#1}^{#2}}

\DeclareMathOperator{\rel}{rel}

\newcommand{\Hi}{\mathcal{H}}
\newcommand{\Hip}{\mathcal{H}_{\pi}}
\newcommand{\G}{\mathcal{G}}
\def\chull#1{\Omega(#1)}
\def\phull#1{\Omega^{\times}(#1)}
\def\dhull#1{\Omega_0(#1)}
\def\deloid#1{\mathcal{G}(#1)}

\title{A dynamical approach to sampling and interpolation in unimodular groups}

\author{Ulrik Enstad}
\address{Department of Mathematics,
University of Oslo,
Moltke Moes vei 35,
0851 Oslo.}
\email{ubenstad@math.uio.no}

\author{Sven Raum}
\address{Institut für Mathematik,
  Universit{\"a}t Potsdam,
  Campus Golm, Haus 9,
  Karl-Liebknecht-Str. 24-25,
  D-14476 Potsdam OT Golm,
  Germany.}
\email{sven.raum@uni-potsdam.de}

\begin{document}

\begin{abstract}
    We introduce a notion of covolume for point sets in locally compact groups that simultaneously generalizes the covolume of a lattice and the reciprocal of the Beurling density for amenable, unimodular groups. This notion of covolume arises naturally from transverse measure theory applied to the hull dynamical system associated to a point set. Using groupoid techniques, we prove necessary conditions for sampling and interpolation in reproducing kernel Hilbert spaces of functions on unimodular groups in terms of this new notion of covolume. These conditions generalize previously known density theorems for compactly generated groups of polynomial growth, while also covering important new examples, in particular model sets arising from cut-and-project schemes.
\end{abstract}

\maketitle

\section{Introduction}

In this paper we consider necessary density conditions for sampling and interpolation in locally compact groups. Roughly speaking, a discrete subset $\Lambda$ of a group $G$ is sampling for a space of functions $\Hi$ on $G$ if every $f \in \Hi$ can be (stably) reconstructed from its restriction $f|_{\Lambda}$ to $\Lambda$. Dually, $\Lambda$ is interpolating if every (suitable) sequence $(c_\lambda)_{\lambda \in \Lambda}$ appears as the restriction of a function $f \in \Hi$, that is, $f(\lambda) = c_\lambda$ for all $\lambda \in \Lambda$.

Necessary density conditions began with Landau's celebrated theorem for Paley--Wiener spaces of band-limited functions on $\R^d$ \cite{La67}. Using the now ubiquitous density concept of Beurling \cite{Be89}, Landau showed that a set of sampling for a Paley--Wiener space necessarily has lower density at least the Lebesgue measure $m(K)$ of the corresponding bandwidth $K \subseteq \R^d$, while a set of interpolation necessarily has upper density at most $m(K)$. Since then, necessary (and sometimes sufficient) density conditions in terms of Beurling densities have been proved in many other contexts, including (but not limited to) Bargmann--Fock spaces \cite{SeWa92-2,Ly92}, Gabor frames \cite{He07,RaSt95} and more general coherent systems \cite{Ho14}, and general reproducing kernel Hilbert spaces on metric measure spaces \cite{FuGr07,MiRa20}.

A Paley--Wiener space is an example of a reproducing kernel Hilbert space of functions on $\R^d$ which is invariant under translation by elements from $\R^d$. More generally, many sampling and interpolation problems in the literature have natural formulations in terms of reproducing kernel Hilbert spaces $\Hi$ on other locally compact groups $G$ which are invariant under left translations: That is, whenever $f \in \Hi$ and $x \in G$, the translated function $\lreg{}{x}f$ given by $\lreg{}{x}f(y) = f(x^{-1}y)$ is in $\Hi$. A main source of examples comes from (projective) discrete series representations \cite{FuGr07}, which includes in particular the theory of Gabor systems and generalized time-frequency shifts \cite{Gr21}, wavelets, and frames in the orbit of Bergman kernels \cite{Pe72,MoPe74}. On the other hand, the Paley--Wiener spaces themselves can be straightforwardly generalized to locally compact abelian groups, and have also seen generalizations to certain non-abelian groups, see e.g.\ \cite{Pe98}.

It is thus a natural question whether Landau's necessary density conditions can be generalized to translation-invariant reproducing kernel Hilbert spaces on general locally compact groups. An immediate obstacle, however, is to give a proper generalization of Beurling density to this setting. In the case of compactly generated groups of polynomial growth, defining Beurling densities in terms of suitable metric balls were shown to give the correct notion of density in \cite[Theorem 5.3]{FuGrHa17}. However, as remarked in \cite{FuGrHa17}, metric balls do not seem to give the correct notion of density e.g.\ in groups of exponential growth.

The goal of this paper is to introduce a notion of density and to prove necessary density conditions for sampling and interpolation in general locally compact, second-countable (lcsc) groups that are \emph{unimodular}: That is, the left and right Haar measures agree. Thus, we include polynomial growth groups and non-amenable examples such as $SL(2,\R)$ and other semisimple Lie groups, but leave out e.g.\ the affine group. For unimodular groups, we define for any relatively separated set $\Lambda \subseteq G$ two numbers $\covol_-(\Lambda)$ and $\covol_+(\Lambda)$ which we term the \emph{lower} and \emph{upper covolume} of $\Lambda$, respectively. As our following main theorem suggests, the reciprocals of these numbers give a notion of upper and lower density, respectively:

\begin{theorem}\label{thm:necessary_conditions}
Let $\Lambda$ be a subset of a unimodular lcsc group $G$ with identity $e$ and let $\Hi$ be a reproducing kernel Hilbert space of functions on $G$ and denote by $k_x$ the kernel at $x \in G$, i.e.,\
\[ f(x) = \langle f, k_x \rangle , \quad f \in \Hi . \]
Suppose that the following assumptions are satisfied:
\begin{enumerate}[label=(\alph*)]
    \item $\Hi$ is isometrically embedded into $L^2(G)$.
    \item $G$-invariance: For every $f \in \Hi$ and $x \in G$, the translation $\lreg{}{x}f(y) = f(x^{-1}y)$ is in $\Hi$.
    \item Localization: For some compact neighborhood $Q$ of the identity of $G$ it holds that
    \[ \int_G \sup_{t \in xQ}|f(t)|^2 \dif{x} < \infty , \quad f \in \Hi .\]
\end{enumerate}
Then the following hold:
\begin{enumerate}
    \item If $\Lambda$ is a set of sampling for $\Hi$, that is, there exist $A,B > 0$ such that
    \[ A \| f \|^2 \leq \sum_{\lambda \in \Lambda} |f(\lambda)|^2 \leq B \| f \|^2, \quad f \in \Hi , \]
    then
    \[ \| k_e \|^2 \covol_+(\Lambda) \leq 1 .\]
    \item If $\Lambda$ is a set of interpolation for $\Hi$, that is, for every $(c_\lambda)_{\lambda \in \Lambda} \in \ell^2(\Lambda)$ there exists $f \in \Hi$ such that
    \[ f(\lambda) = c_\lambda, \]
    then
    \[ \| k_e \|^2 \covol_-(\Lambda) \geq 1 . \]
\end{enumerate}
\end{theorem}

We remark that the numbers $\| k_e \|^2 \covol_+(\Lambda)$ and $\| k_e \|^2 \covol_-(\Lambda)$ in \Cref{thm:necessary_conditions} do not depend on the normalization of a Haar measure on $G$. Interpreting the reciprocal of the lower (resp.\ upper) covolume as an upper (resp.\ lower) density, \Cref{thm:necessary_conditions} states that the critical density for sampling and interpolation is given by $\| k_e \|^2$.

While \Cref{thm:necessary_conditions} is stated for translation invariant reproducing kernel Hilbert spaces, we prove in \Cref{thm:twisted_necessary_conditions} a more general result for invariance under twisted translations with respect to a 2-cocycle on the group $G$. This allows us to cover e.g.\ coherent systems arising from projective discrete series representations (cf.\ \Cref{subsec:coherent}). Such representations are particularly relevant for connected, simply connected, nilpotent Lie groups which do not admit any actual discrete series representations.

\subsection{Covolumes}

Let us detail the definition of the lower and upper covolume. Let $\Lambda$ be a relatively separated set of a unimodular lcsc group $G$, that is, $\rel(\Lambda) \coloneqq \inf_{U}  \sup_{x \in G}|\Lambda \cap xU| < \infty$ where the infimum is taken over all neighborhoods $U$ of the identity. The starting point of our definition is that for a lattice $\Gamma$, its covolume appears naturally in Weil's formula:
\begin{equation}
    \int_{G/\Gamma} \sum_{x \in P} f(x) \dif{\mu(P)} = \covol(\Gamma)^{-1} \int_G f(x) \dif{x} , \qquad f \in C_c(G) . \label{eq:weil_intro}
\end{equation}
Here, $\mu$ denotes the unique $G$-invariant probability measure on the homogeneous space $G/\Gamma$ of cosets of $\Gamma$.

For more general point sets $\Lambda \subseteq G$, a well-studied generalization of the coset space of a lattice is given by the (punctured) hull dynamical system $\phull{\Lambda}$, cf.\ \Cref{sec:point_sets}. In analogy with the lattice case we define the covolume of $\Lambda$ with respect to an invariant probability measure $\mu$ on $\phull{\Lambda}$ to be the number $\covol_\mu(\Lambda)$ appearing in a suitable generalization of Weil's formula, see \eqref{eq:weil} on p.\ \pageref{eq:weil}. Alternatively $\covol_\mu(\Lambda)$ can be defined in terms of the transverse measure of $\mu$, see \Cref{def:full_covolume}. We then set
\begin{align*} \covol_-(\Lambda) = \inf_{\mu} \covol_\mu(\Lambda), && \covol_+(\Lambda) = \sup_\mu \covol_\mu(\Lambda), \end{align*}
where the infimum and supremum are taken over all such invariant measures (if no measures exist, we set $\covol_-(\Lambda) = \infty$ and $\covol_+(\Lambda) = 0$).

Hull dynamical systems of point sets are among the main objects of study in the field of aperiodic order \cite{BaGr13} and its recent generalizations to the realm of non-abelian locally compact groups \cite{BjHa18,BjHaPo18}. Of particular interest are point sets for which the associated punctured hull is uniquely ergodic, i.e.,\ admits a unique invariant probability measure $\mu$, in which case $\covol_-(\Lambda) = \covol_+(\Lambda) = \covol_\mu(\Lambda)$. In many cases, the dynamical property of linear repetivity is known to be a sufficient condition for unique ergodicity, see \cite{LaPl03,BeHaPo20}. Another important class of uniquely ergodic point sets comes from cut-and-project schemes introduced by Meyer in \cite{Me72} and their generalization beyond abelian groups \cite{BjHaPo18}. Specifically, a model set in $G$ is a point set of the form $\Lambda = p_G(\Gamma \cap (G \times W))$ where $\Gamma$ is an appropriate lattice in a product group $G \times H$ for some locally compact group $H$, $W$ is a compact subset of $H$, and $p_G$ denotes the projection from $G \times H$ onto $G$. Under certain regularity assumptions on $W$ it was proved in \cite{BjHaPo18} that model sets in general lcsc groups are uniquely ergodic, and it can be derived from \cite{BjHaPo18} that the covolume of such a regular model set is given by the natural formula
\begin{equation}
    \covol_-(\Lambda) = \covol_+(\Lambda) = \frac{\covol(\Gamma)}{m_H(W)} . \label{eq:covolume_model_set}
\end{equation}
Here $\covol(\Gamma)$ is the covolume of $\Gamma$ in $G \times H$ and $m_H$ denotes the Haar measure on $H$ (cf.\ also \Cref{ex:model_sets}). Hence, for regular model sets, the lower and upper covolume take a particularly simple form.

As an example, consider the group $G = SL(2,\R)$ of $2 \times 2$ matrices with determinant $1$, which acts on the complex upper half-plane $\C^+ = \{ x + iy : y > 0 \}$ via Möbius transformations:
\[ \begin{pmatrix} a & b \\ c & d \end{pmatrix} \cdot z = \frac{az+b}{cz+d}, \quad \begin{pmatrix} a & b \\ c & d \end{pmatrix} \in SL(2,\R), z \in \C^+ .\]
Let $\Hi_k$ denote the Bergman space of analytic functions $f$ on $\C^+$ satisfying
\[ \| f \|_k^2 = \int_{\C^+} |f(x+iy)|^2 y^{k-2} \dif{x} \dif{y} < \infty . \]
A discrete series representation (cf.\ \Cref{subsec:coherent}) of $G = SL(2,\R)$ on $\Hi_k$ is given via
\[ \pi_k(A)f(z) = \frac{1}{(cz+d)^k} f( A^{-1} \cdot z) , \quad A = \begin{pmatrix} a & b \\ c & d \end{pmatrix} . \]
Given any so-called admissible analyzing vector $\eta \in \Hi_k$, the set of matrix coefficients $\{ \langle \xi, \pi_k( \cdot ) \eta \rangle : \xi \in \Hi_k \} \subseteq L^2(G)$ is a reproducing kernel Hilbert space on $G$ satisfying the assumptions of \Cref{thm:necessary_conditions} with critical density given by the formal dimension of $\pi_k$, $d_{\pi_k} = (k-1)/(4\pi)$ \cite[7.18]{Su75}. A set $\Lambda \subseteq G$ is a set of sampling (resp.\ interpolation) if and only if the coherent system $\pi(\Lambda)\eta = (\pi(\lambda)\eta)_{\lambda \in \Lambda}$ is a frame (resp.\ Riesz sequence) for $\Hi_k$.

One can construct model sets $\Lambda$ in $SL(2,\R)$ as follows: Let $\Gamma$ be a cocompact irreducible lattice in $SL(2,\R) \times SL(2,\R)$ and let $W \subseteq SL(2,\R)$ be a compact identity neighborhood.  Then the projection of $\Gamma \cap (SL(2,\R) \times W)$ to the first component is a relatively dense model set. Examples of cocompact irreducible lattices in products of simple Lie groups can be obtained from arithmetic constructions. We refer to \cite[p.\ 1160]{BjHaPo22} for an explicit example. \Cref{thm:necessary_conditions} now states the following for the model set $\Lambda$ obtained above and an admissible analyzing vector $\eta$:
\begin{align*}
    \text{$\pi(\Lambda)\eta$ is a frame} &\implies \frac{m_H(W)}{\covol(\Gamma)} \geq \frac{k-1}{4\pi} , \\
    \text{$\pi(\Lambda)\eta$ is a Riesz sequence} &\implies \frac{m_H(W)}{\covol(\Gamma)} \leq \frac{k-1}{4\pi}.
\end{align*}

\subsection{Relation to Beurling densities}

In unimodular groups that are amenable, Banach densities provide a way to measure asymptotic frequencies of point sets and are commonly employed in ergodic theory and mathematical diffraction theory \cite{RiSt17,BaMo04,Ho95}. Banach densities are defined in terms of certain sequences of subsets called (strong) Følner sequences or van Hove sequences, cf.\ \Cref{subsec:beurling}, whose existence is equivalent to amenability of the ambient group. In compactly generated groups of polynomial growth, sequences of balls coming from a compatible metric form strong Følner sequences \cite{Br14}, and the resulting Banach densities coincide with Beurling densities \cite{Be89}. For this reason we shall stick to the name Beurling densities in the present paper, as was also done recently in \cite{PoRiSt22}. The following theorem shows that for Delone sets in amenable groups, the introduced lower and upper covolume agree with the reciprocal of upper and lower Beurling density, respectively.

\begin{theorem}
  \label{thm:intro-compare-density-covolume-separeted}
  Let $\Lambda \subseteq G$ be a Delone set, i.e.,\ both separated and relatively dense. Then
  \begin{gather*}
  D^{-}(\Lambda) = \frac{1}{\covol_+(\Lambda)}
  \quad \text{ and } \quad
  D^{+}(\Lambda) = \frac{1}{\covol_-(\Lambda)}
  \text{.}
  \end{gather*}
  More generally, if $\Lambda$ is merely a relatively separated set, then
  \[ \rel(\Lambda)^{-1} \cdot D^+(\Lambda) \leq \frac{1}{\covol_-(\Lambda)} \leq D^+(\Lambda) . \]
  If $\Lambda$ is relatively separated set and relatively dense
  , then
  \[ \rel(\Lambda)^{-1} \cdot D^-(\Lambda) \leq \frac{1}{\covol_+(\Lambda)} \leq D^{-}(\Lambda) . \]
\end{theorem}

Hence, for amenable, unimodular groups, \Cref{thm:necessary_conditions} in combination with \Cref{thm:intro-compare-density-covolume-separeted} gives necessary conditions in terms of Beurling densities which generalizes the result \cite[Theorem 5.3]{FuGrHa17} for polynomial growth groups; see \Cref{cor:necessary_conditions_amenable} for a detailed statement. We also note that after the first version of the present paper was uploaded to arXiv, \Cref{thm:necessary_conditions} for the amenable, unimodular case has been deduced through different methods in \cite{Enva22} and \cite{Cava22}.

We stress that one should not expect $D^-(\Lambda) = 1/\covol_+(\Lambda)$ in general. For instance, a non-cocompact lattice has lower and upper covolume given by its usual lattice covolume, while its lower density is $0$.

\subsection{Proof techniques and technical comments}

The proof of \Cref{thm:necessary_conditions} is inspired by the proof of the lattice density theorem for coherent systems given in \cite{Rova22}, compare also \cite{Ja05}. This proof works in two steps: First, one deduces a relation between the frame bounds of a set of sampling, the covolume of $\Gamma$, and the critical density. This is done by first noting that the frame property passes to all left translates of $\Gamma$ and then integrating the frame inequalities over $G/\Gamma$. The second step is to prove that one can always pass from a frame to an associated Parseval frame (i.e.,\ with frame bounds $A=B=1$) of the same form, which in the lattice case follows from the fact that the frame operator commutes with the (projective) left regular representation of $\Gamma$ as an abstract group. Combined with step 1, one can deduce the necessary density conditions.

We will generalize this proof strategy to non-uniform point sets and follow the same overall two-step procedure. In the first step, we integrate the frame bounds over the punctured hull $\phull{\Lambda}$ with respect to an invariant measure. This requires a stability result for frames over limits of point sets in the Chabauty--Fell topology analogous to the weak limits of Beurling \cite{Be89}. The proof strategies are similar to the analogous results for Gabor frames \cite{He08,RaSt95}. The second step is more subtle as one needs a replacement for the left regular representation of lattices. One of the main conceptual tools of the paper to handle this are the notions of frame and Riesz vectors for groupoid representations which we introduce in \Cref{subsec:frame_vectors}. Specifically, we show that $\Lambda$ is a set of sampling (resp.\ interpolation) if and only if the associated kernel function $k_e$ is a frame vector (resp.\ Riesz vector) for an associated representation of a groupoid $\deloid{\Lambda}$. Passing to the associated Parseval frame gives a new frame vector for this groupoid representation, which is a consequence of the fact that the associated frame vector commutes with the left regular representation of the groupoid $\deloid{\Lambda}$.

We term the groupoid $\deloid{\Lambda}$ simply the \emph{groupoid of $\Lambda$.} It is given by the restriction of the transformation groupoid $G \ltimes \phull{\Lambda}$ to a canonical transversal $\dhull{\Lambda}$. For point sets in $G= \R^d$, these groupoids as well as the closely related tiling groupoids have been studied many times in operator algebras, noncommutative geometry and mathematical physics, see e.g.\ \cite{BoMe19,BoMe21,Be85,BeHeZa00,Ke95,Ke97,KePu00}. The idea that coherent frames over non-uniform point sets can be viewed as ``fibered'' over the unit space of a groupoid already appeared in Kreisel's work \cite{Kr15,Kr16}, from which we have drawn inspiration.

\subsection{Structure of the paper}

In \Cref{sec:point_sets} we define the lower and upper covolume of a point set and prove \Cref{thm:intro-compare-density-covolume-separeted}. In \Cref{sec:sampling_interpolation} we prove \Cref{thm:twisted_necessary_conditions} which includes \Cref{thm:necessary_conditions} as a special case. Then in \Cref{sec:examples} we give examples. At the end we have included \Cref{sec:point-set-groupoid} that characterizes when the groupoid of a point set is étale.

\subsection{Acknowledgements}

The first author would like to thank Petter Kjeverud Nyland and Jordy Timo van Velthoven for insightful discussions. The first author also acknowledges support from The Research Council of Norway through project 314048. The second author was supported by the Swedish Research Council through grant number 2018-04243.

\subsection{Notation}

Throughout the paper $G$ denotes a second-countable locally compact group with identity element $e$. We fix a Haar measure $m$ on $G$ and denote integration with respect to this measure by $\int_G \ldots \dif{x}$.

\section{Covolumes of point sets}\label{sec:point_sets}

The purpose of this section is to introduce the lower and upper covolumes of a point set appearing in \Cref{thm:necessary_conditions}, see \Cref{def:full_covolume}. In the context of amenable groups we prove \Cref{thm:intro-compare-density-covolume-separeted} which describes the relation between covolumes and Beurling densities, see \Cref{subsec:beurling}.

\subsection{The Chabauty--Fell topology}\label{subsec:chabauty-fell}

Denote by $\mathcal{C}(G)$ the set of closed subsets of $G$. The \emph{Chabauty--Fell topology} on $\mathcal{C}(G)$ is given by the subbasis consisting of the sets
\begin{align*}
    \mathcal{O}_{K} &= \{ C \in \mathcal{C}(G) : C \cap K = \emptyset \} , \\
    \mathcal{O}^V &= \{ C \in \mathcal{C}(G) : C \cap V \neq \emptyset \} ,
\end{align*}
where $K$ ranges over all compact subsets of $G$ and $V$ ranges over all open subsets of $G$. Alternatively, it is determined by an open neighborhood basis at each $C \in \mathcal{C}(G)$ consisting of the sets
\begin{align*}
    \mathcal{V}_{K,V}(C) = \{ D \in \mathcal{C}(G) : D \cap K \subseteq CV, C \cap K \subseteq DV \}
\end{align*}
where $K$ runs through all compact subsets of $G$ and $V$ runs through all open neighborhoods of the identity, cf.\ \cite[Proposition A.1]{BjHaPo18}. For the rest of the paper, all topological considerations on $\mathcal{C}(G)$ will be with respect to the Chabauty--Fell topology. The space $\mathcal{C}(G)$ is compact. Observe that since $G$ is second-countable, also $\mathcal{C}(G)$ is second-countable, so its topology is described by convergence of sequences.

The following well-known lemma describes convergence of sequences in $\mathcal{C}(G)$ and will occasionally be useful. See for example \cite[Proposition E.1.2]{BePe92} for a proof.

\begin{lemma}\label{lem:chab_conv}
Let $P_n, P \in \mathcal{C}(G)$ for each $n \in \N$. Then $P_n \to P$ if and only if both of the following statements hold:
\begin{enumerate}
    \item Whenever $x \in P$ then there exist $x_n \in P_n$ such that $x_n \to x$.
    \item Whenever $(n_k)_{k \in \N}$ is a subsequence of $\N$ and $x_{n_k} \in P_{n_k}$ with $x_{n_k} \to x \in G$, then $x \in P$.
\end{enumerate}
\end{lemma}

\subsection{Point sets}\label{subsec:points-sets}

Let $\Lambda$ be a subset of $G$. Given $S \subseteq G$ and $\ell \in \N$, we say that $\Lambda$ is
\begin{enumerate}
    \item \emph{$\ell$-relatively $S$-separated} if $|\Lambda \cap xS| \leq \ell$ for all $x \in G$;
    \item \emph{$S$-dense} if $|\Lambda \cap xS| \geq 1$ for all $x \in G$.
\end{enumerate}
For a fixed $\ell \in \N$ we say that $\Lambda$ is \emph{$\ell$-relatively separated} if there exists a nonempty open set $U \subseteq G$ (which can always be taken as a precompact, symmetric neighborhood of the identity) such that $\Lambda$ is $\ell$-relatively $U$-separated. A set is \emph{relatively separated}\footnote{The term \emph{uniformly locally finite} is also used in the literature.} if it is $\ell$-relatively separated for some $\ell \in \N$. Equivalently
\begin{equation}
    \rel_U(\Lambda) \coloneqq \sup_{x \in G}|\Lambda \cap xU| < \infty
\end{equation}
for some (equivalently all) nonempty open precompact sets $U \subseteq G$. We denote by
\begin{equation*}
    \rel(\Lambda) \coloneqq \inf_U \sup_{x \in G}|\Lambda \cap xU| ,
\end{equation*}
the minimal number $\ell = \rel(\Lambda)$ such that $\Lambda$ is $\ell$-relatively separated. If $\Lambda$ is $1$-relatively $U$-separated (resp.\ $1$-relatively separated) we simply say that $\Lambda$ is  \emph{$U$-separated} (resp.\ \emph{separated})\footnote{The term \emph{uniformly discrete} is also used in the literature.}.

On the other hand, if $\Lambda$ is $K$-dense for some compact set $K$, then $\Lambda$ is called \emph{relatively dense}. If $\Lambda$ is both separated and relatively dense, then $\Lambda$ is called a \emph{Delone set}.

\begin{proposition}\label{prop:U_discrete_closed}
Let $k, \ell \in \N$, let $U \subseteq G$ be open and let $K \subseteq G$ be compact. The following statements hold:
\begin{enumerate}
    \item The set $\{ C \in \mathcal{C}(G) : |C \cap U| \geq k \}$ is open in $\mathcal{C}(G)$.
    \item The set of $\ell$-relatively $U$-separated sets is closed in $\mathcal{C}(G)$.
    \item The set of all $K$-dense sets is closed in $\mathcal{C}(G)$.
\end{enumerate}
\end{proposition}

\begin{proof}
(i) Let $C \in \mathcal{C}(G)$ be such that $|C \cap U| \geq k$ so that there exist distinct elements $x^1,\ldots,x^k \in C \cap U$. Let $C_n \to C$ in $\mathcal{C}(G)$. By \Cref{lem:chab_conv} (i) we can find sequences $x_n^j \to x^j$ with $c_n^j \in C_n$ for each $1 \leq j \leq k$. Then for $j \neq j'$ we must have $x_n^j \neq x_n^{j'}$ eventually, and also $x_n^j \in U$ eventually. Hence $|C_n \cap U| \geq k$ eventually, which proves the claim.

(ii) The compliment of the collection of $\ell$-relatively $U$-separated sets is equal to
\[ \bigcup_{x \in G} \{ P \in \mathcal{C}(G) : |P \cap xU| \geq \ell + 1 \} \]
which is open by (i).

(iii) The compliment of the collection of $K$-dense sets is equal to $\bigcup_{x \in G} \mathcal{O}_{xK}$ which is open by definition of the Chabauty--Fell topology.
\end{proof}

The following proposition describes convergence of relatively separated sets. Its proof is similar to that of \cite[Corollary 4.7]{BjHa18}.

\begin{proposition}\label{prop:rel_sep_conv}
Let $(P_n)_{n \in \N}$ be a sequence of $\ell$-relatively $U$-separated sets that converges to a set $P$ in $\mathcal{C}(G)$. Let $K \subseteq G$ be a compact subset of $G$ such that $P \cap \partial K = \emptyset$. Write $P \cap K = \{ x^1, \ldots, x^k \}$. Then there exists an $N \in \N$ such that when $n \geq N$ the set $P_n \cap K$ can be partitioned into $k$ sets
\[ P_n \cap K = M_n^{1} \sqcup \cdots \sqcup M_n^{k} \]
of cardinalities satisfying $1 \leq |M_n^j| \leq \ell$ such that $M_n^j \to \{ x^j \}$ as $n \to \infty$ for each $1 \leq j \leq k$, that is $x_n \to x^j$ for every sequence of elements $x_n \in M_n^j$, $n \in \N$.
\end{proposition}

There is a continuous left $G$-action on $\mathcal{C}(G)$ given by
\[ xC = \{ xy : y \in C \} \;\;\; \text{for $x \in G$ and $C \in \mathcal{C}(G)$} . \]
The compact space given by the closure of the orbit of a set $\Lambda \in \mathcal{C}(G)$ under this action is called the \emph{hull} of $\Lambda$ and is denoted by
\[ \chull{\Lambda} = \overline{ \{ x\Lambda : x \in G \} } \subseteq \mathcal{C}(G) .\]
For $G=\R^d$ this is precisely the set of weak limits of $\Lambda$ studied by Beurling \cite{Be89}. The \emph{punctured hull} of $\Lambda$ is the locally compact space $\phull{\Lambda} = \chull{\Lambda} \setminus \{ \emptyset \}$. Note that $\Lambda$ is relatively dense if and only if $\emptyset \notin \chull{\Lambda}$, so that $\phull{\Lambda} = \chull{\Lambda}$ is compact in this case. 

A useful estimate for arbitrary elements in the hull that can be obtained as a corollary of \Cref{prop:U_discrete_closed} is the following one.

\begin{corollary}\label{cor:rel_sep_cpt}
Let $U$ be a symmetric identity neighborhood in $G$ and let $\Lambda \subseteq G$ be an $\ell$-relatively $U$-separated set.  For every compact set $K \subseteq G$ we have that
\[ |P \cap K| \leq \ell \cdot \frac{m(KU) }{m(U)} \text{,} \quad \text{ for all } P \in \chull{\Lambda}\text{.}  \]
\end{corollary}

\begin{proof}
Let $P \in \chull{\Lambda}$. Since $U$ is symmetric we have that $x \in yU$ if and only if $y \in xU$ for any $x,y \in G$. Using this and the fact that $P$ is $\ell$-relatively separated by \Cref{prop:U_discrete_closed} we obtain for every $x \in G$ that
\[ \sum_{y \in P \cap K} \mathbbm{1}_{yU}(x) = \sum_{y \in P \cap K} \mathbbm{1}_{xU}(y) = |P \cap K \cap xU| \leq |P \cap xU| \leq \ell . \]
Note also that if $x \in yU$ for some $y \in K$ then $x \in KU$. Consequently
\[ m(U)|P \cap K| = \int_G \sum_{y \in P \cap K} \mathbbm{1}_{yU}(x) \dif{x} = \int_{KU} \sum_{y \in P \cap K} \mathbbm{1}_{yU}(x) \dif{x} \leq m(KU) \cdot \ell , \]
and the conclusion follows.
\end{proof}

The \emph{transversal}\footnote{The transversal is also sometimes called the \emph{discrete hull}.} of $\Lambda$ is
\[ \dhull{\Lambda} = \{ P \in \chull{\Lambda} : e \in P \} \subseteq \phull{\Lambda}. \]
Note that the transversal is compact since it is closed in $\chull{\Lambda}$. Alternatively $\dhull{\Lambda}$ is the closure of the set $\{ \lambda^{-1}\Lambda : \lambda \in \Lambda \}$. If $\Gamma$ is a lattice in $G$, then $\dhull{\Gamma} = \{ \Gamma \}$, $\phull{\Gamma} = \{ x \Gamma : x \in G \}$ and $\phull{\Gamma}$ is homeomorphic as a $G$-space to the quotient space $G/\Gamma$.

\subsection{Transverse measures}\label{subsec:transverse}

We will be concerned with finite $G$-invariant Radon measures on $\phull{\Lambda}$ which we will just refer to as finite invariant measures. We let $\mathcal{P}_G(\phull{\Lambda})$ denote the set of $G$-invariant probability measures on $\phull{\Lambda}$. The following proposition will be a key result for us. It is a special case of Connes' transverse measure theory introduced in \cite{connes79}. A proof in the setting of equivalence relations can be found in \cite[Proposition 4.3]{KyPeVa15}. Note that transverse measures for hulls of point sets were also used in \cite{BjHaKa21}.

\begin{proposition}\label{prop:transverse_measure}
Let $\Lambda \subseteq G$ be relatively separated. Then for every non-zero finite invariant measure $\mu$ on $\phull{\Lambda}$ there exists a unique measure $\mu_0$ on $\dhull{\Lambda}$ such that
\begin{equation}
    \int_{\phull{\Lambda}} \sum_{x \in P} F(x,x^{-1}P) \dif{\mu(P)} = \int_{\dhull{\Lambda}} \int_G F(x,Q) \dif{x} \dif{\mu_0(Q)} \label{eq:transverse}
\end{equation}
for all $F \in C_c(G \times \dhull{\Lambda})$. The measure $\mu_0$ is called the \emph{transverse measure} associated to $\mu$ and is finite and non-zero.
\end{proposition}

Note that we do not necessarily assume non-zero finite invariant measures $\mu$ on $\phull{\Lambda}$ to be probability measures, i.e.,\ $\mu(\phull{\Lambda}) = 1$. Indeed, another natural normalization of $\mu$ is to instead assume that the associated transverse measure $\mu_0$ on $\dhull{\Lambda}$ is a probability measure. For $f \in C_c(G)$, the transverse measure formula gives
\[ \mu_0(\dhull{\Lambda}) \int_G f(x) \dif{x} = \int_{\dhull{\Lambda}} \int_G f(x) \dif{x} \dif{\mu_0(Q)} = \int_{\phull{\Lambda}} \sum_{x \in P} f(x) \dif{\mu(P)} .\]
Hence, if $\mu$ is normalized such that $\mu_0(\dhull{\Lambda}) = 1$, the following generalization of \emph{Weil's formula} holds:
\begin{equation} \int_G f(x) \dif{x} = \int_{\phull{\Lambda}} \sum_{x \in P} f(x) \dif{\mu(P)}, \quad \text{for all } f \in C_c(G) . \label{eq:weil} \end{equation}

\begin{definition}\label{def:full_covolume}
Let $\Lambda \subseteq G$ be a relatively separated set and let $\mu$ be a non-zero finite invariant measure on $\phull{\Lambda}$. The \emph{covolume} of $\Lambda$ with respect to $\mu$ is the positive real number
\[ \covol_\mu(\Lambda) = \frac{\mu(\chull{\Lambda})}{\mu_0(\dhull{\Lambda})} .\]
Furthermore, we define the \emph{lower} and \emph{upper covolume} of $\Lambda$ to be the extended real-valued numbers given respectively by
\begin{align*}
    \covol_-(\Lambda) = \inf_{\mu \in \mathcal{P}_G(\phull{\Lambda})} \covol_\mu(\Lambda), && \covol_+(\Lambda) = \sup_{\mu \in \mathcal{P}_{G}(\phull{\Lambda})} \covol_\mu(\Lambda),
\end{align*}
where it is understood that $\covol_-(\Lambda) = \infty$ (resp.\ $\covol_+(\Lambda) = 0$) if $\mathcal{P}_G(\phull{\Lambda}) = \emptyset$.
\end{definition}

\begin{remark}
Note that for any $c > 0$ and any non-zero finite invariant measure $\mu$ on $\phull{\Lambda}$, we have $\covol_{c\mu}(\Lambda) = \covol_\mu(\Lambda)$, hence any particular normalization of $\mu$ does not affect the value of the covolume. If we choose $\mu$ to be a probability measure, then $\covol_\mu(\Lambda) = 1/\mu_0(\dhull{\Lambda})$. On the other hand, if we choose $\mu$ such that Weil's formula \eqref{eq:weil} holds, then $\covol_\mu(\Lambda) = \mu(\phull{\Lambda})$.
\end{remark}

\begin{remark}
The covolume of a point set $\Lambda$ with respect to a finite invariant measure $\mu$ does depend on the choice of Haar measure $m$ on $G$. Indeed, if one rescales the Haar measure by some $c > 0$, then by \eqref{eq:transverse} the transverse measure rescales by $c^{-1}$, so that the covolume rescales by $c$.
\end{remark}

\begin{remark}
Let $\Gamma$ be a lattice in $G$. Then its usual covolume $\covol(\Gamma)$ is given by $\mu(G/\Gamma)$ where $\mu$ is the invariant measure on $G/\Gamma$ normalized such that Weil's formula holds. Since $\phull{\Gamma} = G/\Gamma$, it follows $\covol_\mu(\Gamma) = \covol(\Gamma)$, hence $\covol_-(\Gamma) = \covol_+(\Gamma) = \covol(\Gamma)$.
\end{remark}

In the next proposition we will obtain useful and natural bounds on the covolume.
\begin{proposition}\label{prop:covolume_average}
For any Borel set $S \subseteq G$ with $0 < m(S) < \infty$ and non-zero finite invariant measure $\mu$ on $\phull{\Lambda}$, we have that
\[ \frac{1}{\covol_\mu(\Lambda)} = \frac{1}{\mu(\phull{\Lambda})} \int_{\phull{\Lambda}} \frac{|P \cap S|}{m(S)} \dif{\mu(P)} . \]
In particular, the following statements hold:
\begin{enumerate}
    \item If $\Lambda$ is $K$-dense for some compact subset $K \subseteq G$, then $\covol_+(\Lambda) \leq m(K)$.
    \item If $\Lambda$ is $\ell$-relatively $U$-separated for some $\ell \in \N$ and open set $U \subseteq G$, then $\covol_-(\Lambda) \geq m(U)/\ell$.
\end{enumerate}
\end{proposition}

\begin{proof}
The transverse measure formula \eqref{eq:transverse} gives
\[ \mu_0(\dhull{\Lambda}) m(S) = \int_{\dhull{\Lambda}} \int_G \mathbbm{1}_S(x) \dif{x} \dif{\mu_0(Q)} = \int_{\phull{\Lambda}} \sum_{x \in P} \mathbbm{1}_S(x) \dif{\mu(P)} = \int_{\phull{\Lambda}} |P \cap S| \dif{\mu(P)} . \]
Hence
\begin{equation}
    \frac{1}{\covol_\mu(\Lambda)} = \frac{\mu_0(\dhull{\Lambda})}{\mu(\phull{\Lambda})} = \frac{1}{\mu(\phull{\Lambda})} \int_{\phull{\Lambda}} \frac{|P \cap S|}{m(S)} \dif{\mu(P)} .
\end{equation}
In particular, if $\Lambda$ is $K$-dense, so that $|P \cap K| \geq 1$ for all $P \in \chull{\Lambda}$ by \Cref{prop:U_discrete_closed}, then $\covol_\mu(\Lambda)  \leq m(K)$. On the other hand, if $\Lambda$ is $\ell$-relatively $U$-separated, so that $|P \cap U| \leq \ell$ for all $P \in \chull{\Lambda}$ by \Cref{prop:U_discrete_closed}, then $\covol_\mu(\Lambda) \geq m(U)/\ell. $
\end{proof}

\begin{example}[Model sets]\label{ex:model_sets}
Following the presentation of \cite[Section 2.2]{BjHaPo18}, a \emph{cut-and-project scheme} is a triple $(G,H,\Gamma)$ where $G$ and $H$ are locally compact groups and $\Gamma$ is a lattice in $G \times H$ such that the projection map $p_G \colon G \times H \to G$ is injective when restricted to $\Gamma$ and the image of $\Gamma$ under the projection map $p_H \colon G \times H \to H$ is dense in $H$. If $W$ is a compact subset of $H$, a set of the form
\[ \Lambda = p_G(\Gamma \cap (G \times W)) \]
is called a \emph{weak model set} in $G$. The set $W$ is called the \emph{window}. If $W$ is Jordan-measurable (that is $\partial W$ has zero Haar measure) with dense interior, aperiodic and $\Gamma$-regular (that is $\partial W \cap p_H(\Lambda) = \emptyset$) then $\Lambda$ is called a \emph{regular} model set. For such model sets there exists a unique $G$-invariant probability measure $\mu$ on $\chull{\Lambda}$ \cite[Theorem 3.4]{BjHaPo18}. It follows from \cite[Proposition 4.13]{BjHaPo18} that
\[ \covol_-(\Lambda) = \covol_+(\Lambda) = \covol_\mu(\Lambda) = \frac{\covol(\Gamma)}{m_H(W)}. \]
\end{example}

\subsection{Relation to Beurling density}\label{subsec:beurling}

In this subsection we assume that $G$ is amenable. For such groups, the \emph{lower} and \emph{upper Beurling density} of a set $\Lambda \subseteq G$ are defined respectively as
\begin{align*}
    D^-(\Lambda) = \sup_{K} \inf_{x \in G} \frac{|\Lambda \cap xK|}{m(K)}, && D^+(\Lambda) = \inf_K \sup_{x \in G} \frac{|\Lambda \cap xK|}{m(K)} ,
\end{align*}
where the infimum and supremum are taken over all compact subsets $K$ of $G$ with positive measure.

A sequence $(F_n)_{n \in \N}$ of compact subsets of positive measure in $G$ is called a \emph{(right) strong Følner sequence} if for every compact set $K \subseteq G$, one has
\[ \lim_{n \to \infty} \frac{m(F_nK \cap F_n^cK)}{m(F_n)} = 0 . \]
The lower and upper Beurling densities may alternatively be described in terms of a strong Følner sequence as follows, cf.\ \cite[Proposition 5.14]{PoRiSt22}:
\begin{align*}
    D^-(\Lambda) = \lim_{n \to \infty} \inf_{x \in G} \frac{|\Lambda \cap xF_n|}{m(F_n)}, && D^+(\Lambda) = \lim_{n \to \infty} \sup_{x \in G} \frac{|\Lambda \cap xF_n|}{m(F_n)} .
\end{align*}
Note that for compactly generated groups of polynomial growth, balls defined with respect to a corresponding word metric form a Følner sequence, cf.\ \cite[Cororollary 1.10]{Br14}, so that Beurling densities can be defined with respect to balls as was done in e.g.\ \cite{FuGrHa17}.

It will be of interest to us to consider the following quantities.
\begin{definition}
\label{def:hull-beurling-denisity}
 The \emph{lower} and \emph{upper hull Beurling density} of $\Lambda$ respectively are
\begin{gather*}
    D^{--}(\Lambda) \coloneqq \sup_{K} \inf_{P \in \phull{\Lambda}} \frac{|P \cap K|}{m(K)}, \quad \text{ and } \quad D^{++}(\Lambda) \coloneqq \inf_K \sup_{P \in \phull{\Lambda}} \frac{|P \cap K|}{m(K)} ,
\end{gather*}
where the supremum and infimum are as before taken over all compact sets $K \subseteq G$ of positive measure.
\end{definition}

The following result describes the relationship between Beurling densities and hull Beurling densities.

\begin{proposition}
\label{prop:beurling_chull}
Let $\Lambda$ be a relatively separated subset of an amenable group $G$. Then
\begin{gather*}
    D^{--}(\Lambda) \leq D^-(\Lambda) \leq \rel(\Lambda) \cdot D^{--}(\Lambda), \quad \text{ and } \quad D^+(\Lambda) = D^{++}(\Lambda).
\end{gather*}
In particular, if $\Lambda$ is separated then $D^{--}(\Lambda) = D^-(\Lambda)$.
\end{proposition}

\begin{proof}
Let $\Lambda$ be $\ell$-relatively $U$-separated for some precompact, symmetric neighborhood $U$ of the identity of $G$. Note first that since $\{ x \Lambda : x \in G \} \subseteq \phull{\Lambda}$, we have that $D^{--}(\Lambda) \leq D^-(\Lambda)$ and $D^{++}(\Lambda) \geq D^+(\Lambda)$.

We consider first the lower densities. For a compact subset of positive measure $K \subseteq G$ and $P \in \phull{\Lambda}$, we can write $P \cap K = \{ x_1, \ldots, x_k \}$.  Since $P \in \phull{\Lambda}$, there exists $x \in G$ such that $x\Lambda \cap K \subseteq (P \cap K)U = x_1U \cup \cdots \cup x_kU$. Then
\[ |x\Lambda \cap K| \leq \sum_{j=1}^k|x\Lambda \cap x_jU| \leq k \ell = |P \cap K|\ell .\]
This shows that $\inf_{x \in G}|x\Lambda \cap K| \leq \ell \cdot \inf_{P \in \phull{\Lambda}}|P \cap K|$, hence $D^-(\Lambda) \leq \ell \cdot D^{--}(\Lambda)$.

For the upper densities, let $\epsilon > 0$ and pick a symmetric, precompact neighborhood $V$ of the identity such that the sets $( x_j V )_{j=1}^k$ are pairwise disjoint. Take $K$ and $P$ as before.  Since $P \in \phull{\Lambda}$ there exists $x \in G$ such that $P \cap K \subseteq x\Lambda V$. Hence for each $1 \leq j \leq k$ there exists $\lambda_j \in \Lambda$ such that $x\lambda_j \in x_j V \subseteq (P \cap K)V$.  In particular, $x\lambda_i \neq x \lambda_j$ for $i \neq j$. This shows that $|P \cap K| \leq |x\Lambda \cap KV|$. Now \Cref{cor:rel_sep_cpt} gives
\[ |x\Lambda \cap KV| - |x\Lambda \cap K| = |x\Lambda \cap (KV \setminus K)| \leq \ell \cdot \frac{m((KV \setminus K)U)}{m(U)} \leq \ell \cdot \frac{m(KL \cap K^cL)}{m(U)}  \]
where $L = \overline{VU}$. Combining the two inequalities we have
\[ \sup_{P \in \phull{\Lambda}} \frac{|P \cap K|}{m(K)} \leq \sup_{x \in G} \frac{|x\Lambda \cap K|}{m(K)} + \frac{\ell}{m(U)} \cdot \frac{m(KL \cap K^cL)}{m(K)}\text{.}\]
Since $G$ is amenable, taking infimums over all compact $K \subseteq G$ of positive measure, the Følner property \cite[Proposition 5.4 (iii)]{PoRiSt22} gives
\[ D^{++}(\Lambda) \leq D^{+}(\Lambda) + \frac{\ell}{m(U)} \cdot \inf_{K} \frac{m(KL \cap K^cL)}{m(K)} = D^+(\Lambda) .\]
\end{proof}

It need not be the case that $D^-(\Lambda) = D^{--}(\Lambda)$ for point sets that are not separated. For instance, the point set $\Lambda = \Z \cup \{ k+1/k : k \in \Z \setminus \{ 0 \} \}$ in $\R$ has $D^-(\Lambda) = 2$ but $D^{--}(\Lambda) = 1$.

We are next aiming to prove \Cref{thm:intro-compare-density-covolume-separeted}. To this end, we will need the following lemma.

\begin{lemma}\label{lem:continuity-counting}
For a relatively separated subset $\Lambda$ of $G$ the following hold:
\begin{enumerate}
    \item For any precompact, open neighborhood $U$ of the identity the function $\phull{\Lambda} \to \N$, $P \mapsto |P \cap U|$ is lower semi-continuous.
    \item For any compact set $K \subseteq G$ we have that $\rel(\Lambda)|P \cap K| \geq \limsup_{P' \to P}|P' \cap K|$. In particular, the function $\phull{\Lambda} \to \N$, $P \mapsto |P \cap K|$ is upper semi-continuous whenever $\Lambda$ is separated.
\end{enumerate}
\end{lemma}

\begin{proof}
(i): For $P \in \chull{\Lambda}$ we know that $|P \cap U|$ is finite.  Let $x_1, \dotsc, x_k$ be an enumeration of $P \cap U$ and pick pairwise disjoint open subset $V_1, \dotsc, V_k \subseteq U$ such that $x_i \in V_i$ for all $i$.  Then $|Q \cap U| \geq k$ for all $Q \in \mathcal{O}^{V_1} \cap \cdots \cap \mathcal{O}^{V_k}$.  This shows lower-semicontinuity.

(ii): Suppose $\Lambda$ is $\ell$-relatively $U$-separated. Let $P \cap K = \{ x_1, \ldots, x_k \}$ and let $V \subseteq U$ be a neighborhood of the identity such that the sets $x_jV$, $1 \leq j \leq k$, are pairwise disjoint. If $P_n \to P$ then there exists $N \in \N$ such that when $n \geq N$ we have $P_n \cap K \subseteq (P \cap K)V = \bigcup_{j=1}^k x_jV$. Hence
\[ |P_n \cap K| \leq \sum_{j=1}^k |P_n \cap x_j V | \leq k \cdot \ell = |P \cap K| \ell . \]
This shows the second claim.
\end{proof}

\begin{observation}\label{obs:semicont}
If $(\mu_n)_{n \in \N}$ is a sequence of Borel probability measures on a topological space $\Omega$ that converges in the weak* topology to a measure $\mu$, then it is well-known that $\int_\Omega f \dif{\mu} \geq \limsup_n \int_\Omega f \dif{\mu_n}$ for upper semi-continuous functions $f \colon \Omega \to \R$ bounded from above and $\int_\Omega f \dif{\mu} \leq \liminf_n \int_\Omega f \dif{\mu_n}$ for lower semi-continuous functions $f \colon \Omega \to \R$ bounded from below.

More generally, if $f$ is bounded from above and satisfies the condition $\ell \cdot f(\omega) \geq \limsup_{\omega' \to \omega} f(\omega')$ as in \Cref{lem:continuity-counting}, then $\ell \cdot \int_\Omega f \dif{\mu} \geq \limsup_n \int_\Omega f \dif{\mu_n}$. This can be seen by using the semi-continuity of the auxiliary function $\tilde{f} \colon \Omega \to \R$ given by
\[ \tilde{f}(\omega) = \inf_{U \in \mathcal{U}_\omega} \sup_{\omega' \in U} f(\omega'), \quad \omega \in \Omega , \]
where $\mathcal{U}_\omega$ denotes the set of open neighborhoods of $\omega$.
\end{observation}

The following theorem establishes the exact relationship between hull Beurling densities and covolumes for relatively separated and relatively dense sets.

\begin{theorem}\label{thm:density_covolume}
Let $\Lambda$ be a relatively separated subset of an amenable group $G$. Then
\begin{align*}
    D^{--}(\Lambda) \leq \frac{1}{\covol_+(\Lambda)} , && \rel(\Lambda)^{-1} \cdot D^{++}(\Lambda) \leq \frac{1}{\covol_{-}(\Lambda)} \leq D^{++}(\Lambda) \text{.}
\end{align*}
If $\Lambda$ is relatively dense, then we have
\[ D^{--}(\Lambda) = \frac{1}{\covol_+(\Lambda)} \text{.} \]
\end{theorem}

\begin{proof}
We let $\Lambda$ be $\ell$-relatively $U$-separated for a precompact, symmetric, open neighborhood $U$ of the identity. We will assume that $\phull{\Lambda}$ admits at least one $G$-invariant probability measure, and at the end of the proof show how the theorem statement is also true if no such measures exist.

\textbf{Step 1: Bounding the reciprocal of covolume.}
Note that for any $G$-invariant probability measure $\mu$ on $\phull{\Lambda}$ and any compact set $K \subseteq G$ of positive measure we have that
\[ \inf_{P \in \phull{\Lambda}} \frac{|P \cap K|}{m(K)} \leq \int_{\phull{\Lambda}} \frac{|P \cap K|}{m(K)} \dif{\mu(P)} \leq \sup_{P \in \phull{\Lambda}} \frac{|P \cap K|}{m(K)} .\]
By \Cref{prop:covolume_average} the middle term above equals $1/\covol_\mu(\Lambda)$. Taking infima and suprema over all compact sets $K \subseteq G$ of positive measure in this double inequality, we arrive at
\[ 
D^{--}(\Lambda) \leq \frac{1}{\covol_\mu(\Lambda)} \leq D^{++}(\Lambda)
\text{.}\]
The rest of the proof is occupied by the construction of invariant probability measures $\mu$ for which $1/\covol_\mu(\Lambda)$ attains $D^{--}(\Lambda)$ and $D^{++}(\Lambda)/\ell$, respectively, the attainment of $D^{--}(\Lambda)$ being under the additional assumption that $\Lambda$ is relatively dense.

\textbf{Step 2: Følner sequence estimate.}
Choose a strong Følner sequence $(F_n)_{n \in \N}$ for $G$ and define for each $n \in \N$ and $P \in \chull{\Lambda}$ a probability measure $\mu_{n,P}$ on $\chull{\Lambda}$ via
\[ \int_{\chull{\Lambda}} f \dif{\mu_{n,P}} = \frac{1}{m(F_n)} \int_{F_n} f(x^{-1}P) \dif{x}  \]
for nonnegative Borel functions $f$ on $\chull{\Lambda}$. In other words, $\mu_{n,P}$ is the convolution the Dirac measure $\delta_{P}$ on $\chull{\Lambda}$ with the measure $1/m(F_n) \cdot \mathbbm{1}_{F_n^{-1}} \cdot m$ on $G$. We will show that
\begin{equation}
    \lim_{n \to \infty} \Big| \int_{\chull{\Lambda}} \frac{|P' \cap U|}{m(U)} \dif{\mu_{n,P}}(P') - \frac{|P \cap F_n|}{m(F_n)} \Big| = 0 \;\; \text{uniformly in $P$.} \label{eq:uniform_conv}
\end{equation}

Fix $P \in \chull{\Lambda}$. For $x \in F_n$ we claim that
\begin{equation}
    |P \cap xU| = \sum_{y \in P \cap F_nU} \mathbbm{1}_{F_n \cap yU}(x) . \label{eq:count}
\end{equation}
Indeed, if $y \in P \cap xU$ then $y \in P \cap F_nU$ and $x^{-1}y \in U$ which implies $y^{-1}x \in U^{-1}=U$, hence $x \in yU$. Thus $\mathbbm{1}_{F_n \cap yU}(x) = 1$. On the other hand, if $\mathbbm{1}_{F_n \cap yU}(x) = 1$ for some $y \in P \cap F_nU$ then $x \in yU$ so $y \in xU$, hence $y \in P \cap xU$.

Using \eqref{eq:count}, we obtain
\begin{equation}  \label{eq:counting-integral-upper-estimate}
\begin{aligned}
     \int_{\chull{\Lambda}} |P' \cap U| \dif{\mu_{n,P}}(P') &= \frac{1}{m(F_n)} \int_{F_n} |P \cap xU| \dif{x} \\
    &= \frac{1}{m(F_n)}  \sum_{y \in P \cap F_nU} \int_{G} \mathbbm{1}_{F_n \cap yU}(x) \dif{x} \\
    &\leq \frac{1}{m(F_n)}  \sum_{y \in P \cap F_nU} \int_{G} \mathbbm{1}_{yU}(x) \dif{x} \\
    &= m(U)\frac{|P \cap F_nU|}{m(F_n)}.
\end{aligned}
\end{equation}
On the other hand, using the fact that $\mathbbm{1}_{F_n \cap yU}(x) = \mathbbm{1}_{yU}(x)$ for $x \in F_n \setminus F_n^cU$ we get
\begin{align*}
    \int_{\chull{\Lambda}} |P' \cap U| \dif{\mu_{n,P}}(P') &= \frac{1}{m(F_n)}  \sum_{y \in P \cap F_nU} \int_{G} \mathbbm{1}_{F_n \cap yU}(x) \dif{x} \\
    &\geq \frac{1}{m(F_n)}  \sum_{y \in P \cap (F_n \setminus F_n^cU)} \int_{G} \mathbbm{1}_{yU}(x) \dif{x} \\
    &= m(U) \frac{|P \cap (F_n \setminus F_n^cU)|}{m(F_n)} .
\end{align*}
Combining these two inequalities and writing $K = \overline{U}$, we arrive at
\[ \frac{|P \cap (F_n \setminus F_n^cK)|}{m(F_n)} \leq \frac{1}{m(U)} \int_{\chull{\Lambda}} |P' \cap U| \dif{\mu_{n,P}(P')} \leq \frac{|P \cap F_{n}K|}{m(F_n)} \text{.} \]
In order to compare the left and right hand side above, we apply \Cref{cor:rel_sep_cpt} to find that
\[ |P \cap (F_nK \cap F_n^cK)| \leq C \cdot m((F_nK \cap F_n^cK)U) \leq C \cdot m(F_nK^2 \cap F_n^c K^2)  \]
where $C = \ell/m(U)$. Hence
\begin{align*}
    0 &\leq \frac{|P \cap F_n K|}{m(F_n)} - \frac{|P \cap (F_n \setminus F_n^c K)|}{m(F_n)}
    \\ &= \frac{|P \cap (F_n K \setminus (F_n \setminus F_n^cK) )|}{m(F_n)} \\
    &= \frac{ |P \cap (F_nK \cap F_n^cK)| }{m(F_n)} \\
    &\leq C \cdot \frac{ m(F_nK^2 \cap F_n^cK^2)}{m(F_n)} .
\end{align*}
By definition of a strong Følner sequence the final expression above goes to zero as $n$ tends to infinity. Since it is independent of $P$ and the term $|P \cap F_n|/m(F_n)$ is also bounded between the same two expressions, we have shown that \eqref{eq:uniform_conv} holds.

\noindent \textbf{Step 3: The lower density.} For this step we assume that $\Lambda$ is relatively dense, so that $\chull{\Lambda} = \phull{\Lambda}$. For each $n \in \N$ let $P_n \in \chull{\Lambda}$ be such that
$|P_n \cap F_n| = \inf_{P \in \chull{\Lambda}} |P \cap F_n|$. It then follows from the definition of the lower hull Beurling density that
\[ D^{--}(\Lambda) \geq \liminf_{n \to \infty} \frac{|P_n \cap F_n|}{m(F_{n})} .\]
Combining this with the uniform convergence in \eqref{eq:uniform_conv} we get
\begin{equation}
     D^{--}(\Lambda) \geq \liminf_{n \to \infty} \int_{\chull{\Lambda}} \frac{|P' \cap U|}{m(U)} \dif{\mu_{n,P_n}}(P') \text{.} \label{eq:lower_density_limit}
\end{equation}
Passing to a subsequence, we may assume by the Banach--Alaoglu theorem that $(\mu_{n,P_n})_n$ converges in the weak* topology to a measure $\mu$ in the unit ball of the dual space of $C(\chull{\Lambda})$, which is necessarily a probability measure due to the compactness of $\chull{\Lambda}$. We claim that $\mu$ is $G$-invariant. Indeed, letting $f \in C(\chull{\Lambda})$ and letting $y \in G$ we obtain
\begin{align*}
    & \Big| \int_{\chull{\Lambda}} f(P) \dif{\mu_{n,P_n}(P)} - \int_{\chull{\Lambda}} f(y^{-1}P) \dif{\mu_{n,P_n}(P)} \Big| \\
    &= \frac{1}{m(F_{n})} \Big| \int_{F_n}f(x^{-1}P_n) \dif{x} - \int_{F_n y}f(x^{-1}P_n) \dif{x} \Big| \\
    &\leq \frac{1}{m(F_{n})} \int_{F_n \triangle F_n y} |f(x^{-1}P_n)| \dif{x} \\
    &\leq \sup_{P \in \chull{\Lambda}}|f(P)| \cdot \frac{m(F_{n} \triangle F_{n} y)}{m(F_{n})} .
\end{align*}
Since $(F_n)_n$ is a Følner sequence, the last term above tends to $0$ as $k \to \infty$, which shows that $\mu$ is $G$-invariant.

By \Cref{lem:continuity-counting}, the function the function $\chull{\Lambda} \to [0, \infty)$, $P' \mapsto |P' \cap U|$ is lower semi-continuous.  Since it is also bounded below, the weak* convergence $\mu_{n,P_n} \to \mu$ combined with the inequality \eqref{eq:lower_density_limit} gives that
\[ \frac{1}{\covol_\mu(\Lambda)} = \int_{\chull{\Lambda}} \frac{|P' \cap U|}{m(U)} \dif{\mu(P')} \leq \liminf_{n \to \infty} \int_{\chull{\Lambda}} \frac{|P' \cap U|}{m(U)} \dif{\mu_{n,P_n}(P')} \leq D^{--}(\Lambda)\text{.} \]
Since $D^{--}(\Lambda) \leq 1/\covol_+(\Lambda)$ was shown in Step 1, we conclude that $D^{--}(\Lambda) = 1/\covol_+(\Lambda)$.

\textbf{Step 4: The upper density.}
We assume as before that $\Lambda$ is $\ell$-relatively $U$-separated for a relatively compact neighborhood $U$ of the identity. Set $K = \overline{U}$. We assume without loss of generality that $m(\partial U) = 0$, from which it follows that $m(U) = m(K)$. As in Step 3 we can find elements $Q_n \in \phull{\Lambda}$ such that
\[ D^{++}(\Lambda) = \lim_{n \to \infty} \frac{|Q_n \cap F_n|}{m(F_n)} . \]
We let $\nu$ be a weak* limit of $(\mu_{n,Q_n})_{n \in \N}$ which is $G$-invariant by the same argument as in Step 3. Using again Step 2, \Cref{lem:continuity-counting} (ii) and \Cref{obs:semicont} we can write
\begin{align*}
    D^{++}(\Lambda) &=  \lim_n \int_{\chull{\Lambda}} \frac{|P' \cap U|}{m(U)} \dif{\mu_{n,Q_n}(P')} \\
    &\leq \lim_n \int_{\chull{\Lambda}} \frac{|P' \cap K|}{m(K)} \dif{\mu_{n,Q_n}(P')} \\
    &\leq \ell \cdot \int_{\chull{\Lambda}} \frac{|P' \cap K|}{m(K)} \dif{\nu(P')} \\
    &= \ell \cdot \int_{\phull{\Lambda}} \frac{|P' \cap K|}{m(K)} \dif{\nu(P')}
\end{align*}
Note that $\nu$ is a measure on $\chull{\Lambda}$ and not necessarily on $\phull{\Lambda}$, a problem that we will circumvent as follows: First, if $D^{++}(\Lambda) = 0$ then $0 \leq 1/\covol_-(\Lambda) \leq D^{++}(\Lambda) = 0$, in which case there is nothing to prove. Assume therefore that $D^{++}(\Lambda) > 0$. It then follows from the above estimate that $c \coloneqq \nu(\{ \emptyset \}) < 1$. But then we can define a new measure $\nu'$ on $\chull{\Lambda}$ via
\[ \nu'(S) = (1-c)^{-1} \nu(S \setminus \{ \emptyset \}), \quad \text{$S \subseteq \chull{\Lambda}$ Borel,} \]
which is easily seen to be an invariant probability measure such that $\nu'(\{ \emptyset \}) = 0$, hence it can be considered an invariant probability measure on $\phull{\Lambda}$. Using the above estimate, we can bound the reciprocal of the covolume of $\Lambda$ with respect to $\nu'$ as follows:
\[ \frac{1}{\covol_{\nu'}(\Lambda)} = \int_{\phull{\Lambda}} \frac{|P' \cap K|}{m(K)} \dif{\nu'(P')} = \frac{1}{1-c} \int_{\phull{\Lambda}} \frac{|P' \cap K|}{m(K)} \dif{\nu(P')} \geq D^{++}(\Lambda) / \ell . \]
From this it follows that $1/\covol_-(\Lambda) \geq D^{++}(\Lambda) / \ell$.

Finally, if $\phull{\Lambda}$ admits no $G$-invariant probability measures then $\covol_{+}(\Lambda) = 0$, so the inequality $D^{--}(\Lambda) \leq 1/\covol_+(\Lambda)$ trivially holds. Since we have already shown that $\phull{\Lambda}$ admits a $G$-invariant probability measure when $D^{++}(\Lambda) > 0$, we conclude that $D^{++}(\Lambda) = 0$. This gives the other inequality $\rel(\Lambda)^{-1} \cdot D^{++}(\Lambda) \leq 1/\covol_-(\Lambda) \leq D^{++}(\Lambda)$ which finishes the proof.
\end{proof}

The inequality $D^{++}(\Lambda) \leq \rel(\Lambda)/\covol_-(\Lambda)$ from \Cref{thm:density_covolume} cannot be improved in general. Indeed, the already mentioned point set $\Lambda = \Z \cup \{ k+1/k : k \in \Z \setminus \{ 0 \} \}$ in $\R$ has $1/\covol_-(\Lambda) = 1$ but $D^{++}(\Lambda) = 2$.

\begin{proof}[Proof of \Cref{thm:intro-compare-density-covolume-separeted}]
Follows from a combination of \Cref{prop:beurling_chull} and \Cref{thm:density_covolume}.
\end{proof}

\section{Sampling and interpolation in reproducing kernel Hilbert spaces}\label{sec:sampling_interpolation}

In this section we will prove \Cref{thm:twisted_necessary_conditions} which is an extension of \Cref{thm:necessary_conditions} to twisted translations. The key new concept of the proof is the abstract notion of frame and Riesz vectors for groupoid representations, which we introduce \Cref{subsec:frame_vectors}. We then apply this notion in \Cref{subsec:necessary-density} to prove \Cref{thm:twisted_necessary_conditions}.

\subsection{Groupoids and their projective unitary representations}\label{subsec:groupoids}

In this subsection we briefly introduce the terminology needed for groupoids and their unitary projective representations, see e.g.\ \cite{Re80} for a comprehensive reference.

A \emph{groupoid} is a set $\mathcal{G}$ together with a distinguished subset $\mathcal{G}^{(2)} \subseteq \G \times \G$, a multiplication map $(\alpha,\beta) \mapsto \alpha \beta$ from $\G^{(2)}$ to $\G$ and an inversion map $\alpha \mapsto \alpha^{-1}$ from $\mathcal{G}$ to $\mathcal{G}$ such that the following axioms are satisfied:
\begin{enumerate}
    \item $(\alpha^{-1})^{-1} = \alpha$ for all $\alpha \in \G$;
    \item if $(\alpha,\beta),(\beta,\gamma) \in \G^{(2)}$, then $(\alpha\beta,\gamma), (\alpha,\beta\gamma) \in \G^{(2)}$, and $(\alpha\beta)\gamma = \alpha(\beta\gamma)$; and
    \item $(\alpha,\alpha^{-1}) \in \G^{(2)}$ for all $\alpha \in \G$ and for all $(\alpha,\beta) \in \G^{(2)}$ we have that $\alpha^{-1}(\alpha\beta) = \beta$ and $(\alpha\beta)\beta^{-1} = \alpha$.
\end{enumerate}
The set $\G^{(0)} = \{ \alpha^{-1}\alpha : \alpha \in \G \} = \{ \alpha\alpha^{-1} : \alpha \in \G \}$ is called the \emph{unit space} of the groupoid and the maps $s,r \colon \G \to \G^{(0)}$ given by $s(\alpha) = \alpha^{-1}\alpha$ and $r(\alpha) = \alpha\alpha^{-1}$ are called the \emph{source} and \emph{range} map, respectively. Given $\omega \in \G^{(0)}$, we write $\G_\omega = \{ \alpha \in \G : s(\alpha) = \omega \}$ and $\G^\omega = \{ \alpha \in \G : r(\alpha) = \omega \}$ for the source and the range fibre at $\omega$, respectively.

A \emph{measurable groupoid} is a groupoid $\mathcal{G}$ which is also a measurable space such that $\G^{(0)}$ is a measurable set and the maps $s$, $r$, $(\alpha,\beta) \mapsto \alpha \beta$ and $\alpha \mapsto \alpha^{-1}$ are measurable (here we equip $\G \times \G$ with the $\sigma$-algebra generated by products of measurable sets in $\G$). If $\G$ is standard Borel as a measurable space, then $\G$ is called a \emph{Borel groupoid}, and if the fiber $\G_\omega$ (equivalently $\G^{\omega}$) is countable for each $\omega \in \G^{(0)}$, then $\G$ is called \emph{fibrewise countable}.

Let $\G$ be a measurable groupoid with unit space $\Omega = \G^{(0)}$. A measurable map $\sigma \colon \G^{(2)} \to \T$ is called a \emph{2-cocycle} on $\G$ if
\begin{enumerate}
    \item $\sigma(\alpha,\beta)\sigma(\alpha\beta,\lambda) = \sigma(\alpha,\beta\lambda)\sigma(\beta,\lambda)$ for all $\alpha,\beta,\lambda \in \G$ with $s(\alpha) = r(\beta)$ and $s(\beta) = r(\lambda)$;
    \item $\sigma(\alpha,s(\alpha)) = \sigma(r(\alpha),\alpha) = 1$ for all $\alpha \in \G$.
\end{enumerate}

A \emph{$\sigma$-projective (weakly measurable) unitary representation $\pi$} of $\G$ on a measurable field $(\Hi_{\omega})_{\omega \in \Omega}$ of Hilbert spaces over $\Omega$ (cf.\ e.g.\ \cite[p.\ 269]{Ta02}) is given by unitary maps $\pi(\alpha) \colon \Hi_{s(\alpha)} \to \Hi_{r(\alpha)}$ for every $\alpha \in \G$ such that the following properties are satisfied:
\begin{enumerate}
    \item For every pair of measurable sections $\xi$ and $\eta$ of $(\Hi_{\omega})_{\omega \in \Omega}$, the map $\G \to \C$ given by $\alpha \mapsto \langle \pi({\alpha}) \xi_{s(\alpha)} , \eta_{r(\alpha)} \rangle$ is measurable; and
    \item $\pi({\alpha}) \pi({\beta}) = \sigma(\alpha,\beta) \pi({\alpha \beta})$ whenever $s(\alpha) = r(\beta)$.
\end{enumerate}

Given two $\sigma$-projective unitary representations $\pi$ and $\rho$ of $\G$ on measurable fields $(\Hi_\omega)_{\omega \in \Omega}$ and $(\mathcal{K}_{\omega})_{\omega \in \Omega}$ respectively, a measurable field $(T_{\omega})_{\omega \in \Omega}$ of bounded linear operators from $(\mathcal{H}_\omega)_\omega$ to $(\mathcal{K}_\omega)_\omega$ is said to \emph{intertwine} $\pi$ and $\rho$ if
\[  T_{r(\alpha)} \pi(\alpha) = \rho(\alpha) T_{s(\alpha)} \;\;\; \text{for all $\alpha \in \G$.} \]
If $T$ has the same domain and codomain and $T_{r(\alpha)} \pi(\alpha) = \pi(\alpha) T_{s(\alpha)}$ for all $\alpha \in \G$, we simply say that $T$ intertwines $\pi$.

If $\G$ is a fibrewise countable Borel groupoid, there is a canonical $\sigma$-projective unitary representation associated to $\G$, namely its \emph{$\sigma$-projective left regular representation} $\Ell_\sigma$, defined as follows: Consider the field of Hilbert spaces $(\ell^2(\G^{\omega}))_{\omega \in \Omega}$. Note that an element ${\xi} = (\xi_\omega)_{\omega} \in \prod_{\omega \in \Omega} \ell^2(\G^{\omega})$ defines a function $\G \to \C$ by $\alpha \mapsto \xi_{r(\alpha)}(\alpha)$. We declare the measurable sections of the field $(\ell^2(\G^{\omega}))_{\omega \in \Omega}$ to be those ${\xi}$ for which the associated function on $\G$ is measurable. This gives $(\ell^2(\G^{\omega}))_{\omega \in \Omega}$ the structure of a measurable field of Hilbert spaces over $\Omega$. We now define the $\sigma$-projective left regular representation $\Ell_\sigma$ of $\G$ on this field by defining $\Ell_\sigma(\alpha) \colon \ell^2(\G^{s(\alpha)}) \to \ell^2(\G^{r(\alpha)})$ via
\begin{equation}
    \Ell_\sigma(\alpha) f(\beta) = \sigma(\alpha,\alpha^{-1}\beta) f(\alpha^{-1}\beta) \;\;\; \text{for $f \in \ell^2(\G^{s(\alpha)})$ and $\beta \in \G^{r(\alpha)}$.}
\end{equation}

\subsection{Frame and Riesz vectors}\label{subsec:frame_vectors}

For a detailed account on frame theory, see \cite{Ch03}. Throughout this subsection we fix a fibrewise countable Borel groupoid $\G$ with unit space $\mathcal{G}^{(0)} = \Omega$, a measurable 2-cocycle $\sigma$ and a $\sigma$-projective unitary representation $\pi$ of $\G$ on a measurable field $(\Hi_{\omega})_{\omega \in \Omega}$ of Hilbert spaces.

Given a measurable section $\eta = (\eta_\omega)_\omega$ of $(\Hi_{\omega})_{\omega}$, we can for each $\omega \in \Omega$ consider the $\G^{\omega}$-indexed family
\begin{equation}
    \pi(\G^\omega)\eta \coloneqq (\pi(\alpha) \eta_{s(\alpha)} )_{\alpha \in \G^{\omega}}
\end{equation}
of vectors in $\Hi_\omega$. In this context we make the following definition:

\begin{definition} \hfill
\begin{enumerate}
    \item We say that $\eta$ is a \emph{frame vector for $\pi$} if there exist $0 < A \leq B < \infty$ such that
\begin{equation}
     A \| \xi \|^2 \leq \sum_{\alpha \in \G^{\omega}} | \langle \xi, \pi(\alpha) \eta_{s(\alpha)} \rangle |^2 \leq B \| \xi \|^2 \;\;\; \text{for all $\omega \in \Omega$ and all $\xi \in \Hi_\omega$.} \label{eq:frame_vector}
\end{equation}
In other words, each family $\pi(\G^\omega) \eta$ is a frame for $\Hi_\omega$, with frame bounds $0 < A \leq B < \infty$ independent of $\omega$.
    \item We say that $\eta$ is a \emph{Riesz vector for $\pi$} if there exist $0 < A \leq B < \infty$ such that
    \begin{equation}
        A \| c \|_2^2 \leq \Big\| \sum_{\alpha \in \G^{\omega}} c_{\alpha} \pi(\alpha) \eta_{s(\alpha)} \Big\|^2 \leq B \| c \|_2^2 \;\;\; \text{for all $\omega \in \Omega$ and all $c = (c_{\alpha})_{\alpha \in \G^{\omega}} \in \ell^2(\G^{\omega})$.} \label{eq:riesz_vector}
    \end{equation}
In other words, each family $\pi(\G^\omega) \eta$ is a Riesz sequence for $\Hi_{\omega}$, with Riesz bounds $0 < A \leq B < \infty$ independent of $\omega$.
\end{enumerate}
\end{definition}

If only the upper frame bounds exist in the definition of a frame vector, we call it a \emph{Bessel vector} for $\pi$. We make analogous definitions for Parseval frame vectors and orthonormal vectors.

\begin{proposition}\label{prop:groupoid_bessel_intertwiner}
Suppose $\eta$ is a Bessel vector for $\pi$ and denote by $C_{\omega} \colon \Hi_\omega \to \ell^2(\G^\omega)$ the analysis operator of $\pi(\G^{\omega})\eta$ for each $\omega \in \Omega$, that is,
\[ C_\omega \xi = (\langle \xi, \pi(\alpha) \eta \rangle )_{\alpha \in \G^\omega} , \quad \xi \in \Hi_\omega . \]
Then $C = (C_{\omega})_{\omega \in \Omega}$ is a bounded measurable field of bounded linear operators from $(\Hi_\omega)_{\omega \in \Omega}$ to $(\ell^2(\G^\omega))_{\omega \in \Omega}$ that intertwines $\pi$ and $\Ell_\sigma$.
\end{proposition}

\begin{proof}
That $C$ defines a measurable field follows from the definition of a weakly measurable groupoid representation. Furthermore $C$ is bounded because of the uniform Bessel bound of the families $\pi(\G^{\omega}) \eta$, $\omega \in \Omega$. To prove the intertwining relation, note that for $\alpha,\beta \in \G$ with $r(\alpha) = r(\beta)$ we have
\[ \sigma(\alpha,\alpha^{-1}) = \sigma(\alpha, \alpha^{-1}) \sigma(\alpha\alpha^{-1},\beta) = \sigma(\alpha,\alpha^{-1}\beta)\sigma(\alpha^{-1},\beta) .\]
Hence, letting $\alpha \in \G$, $\xi \in \Hi_{s(\alpha)}$ and $\beta \in \G^{r(\alpha)}$, we get
\begin{align*}
    C_{r(\alpha)} \pi(\alpha) \xi(\beta) &= \langle \pi(\alpha) \xi, \pi(\beta) \eta_{s(\beta)} \rangle_{r(\alpha)} \\
    &= \langle \xi, \pi(\alpha)^*\pi(\beta) \eta_{s(\beta)} \rangle_{r(\alpha)} \\
    &= \langle \xi, \overline{ \sigma(\alpha,\alpha^{-1})} \sigma(\alpha^{-1},\beta) \pi(\alpha^{-1}\beta) \eta_{s(\beta)} \rangle_{r(\alpha)} \\
    &= \sigma(\alpha,\alpha^{-1}\beta) \langle \xi, \pi(\alpha^{-1}\beta) \eta_{s(\beta)} \rangle_{r(\alpha)} \\
    &= \Ell_\sigma(\alpha) C_{s(\alpha)} \xi(\beta) .
\end{align*}
\end{proof}

\begin{corollary}\label{cor:frame_iff_parseval}
The following statements hold:
\begin{enumerate}
    \item \label{it:frame_iff_parseval-frame} If $\pi$ admits a frame vector, then $\pi$ admits a Parseval frame vector.
    \item \label{it:frame_iff_parseval-riesz} If $\pi$ admits a Riesz vector, then $\pi$ admits an orthonormal vector.
\end{enumerate}
\end{corollary}

\begin{proof}
\ref*{it:frame_iff_parseval-frame}: Let $\eta$ be a frame vector for $\pi$ and denote by $S_\omega$ the frame operator of $\pi(\G^\omega)\eta$ for $\omega \in \G^{(0)}$. Then $S_\omega = C_\omega^* C_\omega$ where $C_\omega$ is the analysis operator of $\pi(\G^\omega)\eta$, hence from \Cref{prop:groupoid_bessel_intertwiner} we get that $S = (S_\omega)_\omega$ is a bounded measurable field of bounded linear operators on $(\Hi_\omega)_\omega$ that intertwines $\pi$. By the continuous functional calculus it follows that the field $S^{-1/2} = (S_\omega^{-1/2})_\omega$ intertwines $\pi$. Hence the canonical Parseval frame associated to each $\pi(\G^\omega)\eta$ is of the form
\[ S_\omega^{-1/2} \pi(\G^\omega) \eta = \pi(\G^\omega) S^{-1/2} \eta . \]
This shows that $S^{-1/2}\eta$ is a Parseval frame vector for $\pi$.

\ref*{it:frame_iff_parseval-riesz}: Let $\eta$ be a Riesz vector associated to $\pi$. For each $\omega \in \Omega$ set $\mathcal{K}_{\omega} = \clspn \pi(\G^\omega)\eta$. We then get a representation $\pi|_{\mathcal{K}}$ by restricting each $\pi(\alpha)$ to $\mathcal{K}_{s(\alpha)}$, and $\eta$ becomes both a Riesz vector and a frame vector for $\pi|_{\mathcal{K}}$, so we can form the associated frame operators $S_\omega$ on $\mathcal{K}_{\omega}$. Arguing similarly as in \ref*{it:frame_iff_parseval-frame} it now follows that the field $S^{-1/2} = (S_\omega^{-1/2})_\omega$ intertwines $\pi|_{\mathcal{K}}$, hence $S^{-1/2}\eta$ is an orthonormal vector for $\pi|_{\mathcal{K}}$, hence also for $\pi$.
\end{proof}

\subsection{Sampling and interpolation in reproducing kernel Hilbert spaces}

Throughout this subsection we fix a reproducing kernel Hilbert space (RKHS) $\Hi$ on $G$. That is, $\Hi$ is a Hilbert space of functions on $G$ such that the point evaluations $\Hi \to \C$, $f \mapsto f(x)$ for $x \in G$ are continuous in the norm on $\Hi$. It follows that for each $x \in G$ there exists a vector $k_x \in \Hi$ such that $f(x) = \langle f, k_x \rangle$ for $f \in \Hi$. Furthermore, the \emph{kernel} of $\Hi$ is the function $k \colon G \times G \to \C$ given by
\[ k(x,y) = \langle k_y, k_x \rangle , \quad x,y \in G .\]

We also fix a \emph{(continuous) 2-cocycle} on $G$, that is, a continuous function $\sigma \colon G \times G \to \T$ satisfying the following properties:
\begin{enumerate}
    \item $\sigma(x,y)\sigma(xy,z) = \sigma(x,yz)\sigma(y,z)$ for all $x,y,z \in G$; and
    \item $\sigma(e,e) = 1$.
\end{enumerate}

\begin{assumption}\label{ass:rkhs}
We will make the following assumptions on the RKHS $\Hi$:
\begin{enumerate}[label=(\alph*)]
    \item $\Hi$ is isometrically embedded into $L^2(G)$.
    \item $\sigma$-twisted $G$-invariance: Whenever $f \in \Hi$ and $x \in G$ then $\lreg{\sigma}{x}f \in \Hi$, where
    \[ \lreg{\sigma}{x}f(y) = \sigma(x,x^{-1}y)f(x^{-1}y), \qquad f \in L^2(G) . \]
    \item Localization: For some compact identity neighborhood $Q$ in $G$ we have that
    \begin{equation}
        \int_G \sup_{ t \in xQ} |f(t)|^2 \dif{x} < \infty , \quad f \in \Hi . \label{eq:wiener}
    \end{equation}
\end{enumerate}
\end{assumption}

The twisted $G$-invariance of $\Hi$ significantly simplifies the structure of the kernel of $\Hi$. It implies that $k_x = \lreg{\sigma}{x} k_e$ for all $x \in G$, as shown by the computation
\[ \langle f, k_x \rangle = f(x) = (\lreg{\sigma}{x}^*f)(e) = \langle f, \lreg{\sigma}{x}k_e \rangle , \quad f \in \Hi . \]
Consequently $\| k_x \| = \| k_e \|$ for all $x \in G$. Moreover, the kernel $k$ is determined by the vector $k_e$, since
\[ k(x,y) = \langle k_y, k_x \rangle = \langle k_e, \lreg{\sigma}{y}^*\lreg{\sigma}{x} k_e \rangle = \sigma(y,y^{-1} ) \overline{ \sigma(y^{-1},x) }k_e(y^{-1}x) , \quad x,y \in G. \]
In particular $k(x,x) = k_e(e) = \langle k_e, k_e \rangle = \| k_e \|^2$.

A consequence of the isometric embedding of $\Hi$ into $L^2(G)$ and $\sigma$-twisted $G$-invariance is that every element of $\Hi$ is a continuous function on $G$. This follows from the strong continuity of $\sigma$-twisted left translation on $G$ together with the estimate
\[ |f(x) - f(y)| = |\langle f, k_x - k_y \rangle| \leq \| f \| \| k_x - k_y \| = \| f \| \| \lreg{\sigma}{x} k_e - \lreg{\sigma}{y} k_e \| , \quad f \in \Hi . \]

The equation \eqref{eq:wiener} describes exactly elementhood of a Wiener amalgam space which usually goes by the name of $W(L^2,L^\infty)(G)$ in the literature, cf.\ \cite{Fe83}. We will denote it by $W^2(G)$. The main property of $W^2(G)$ we need is the following lemma, which can be deduced from \cite[Lemma 1]{Gr08} or proved using similar techniques as in \cite[Theorem 11.1.4]{Gr01}.

\begin{lemma}\label{cor:wiener_summation}
Let $f \in W^2(G)$ be continuous. Then for every $\epsilon > 0$ and every $\delta > 0$, there exists a compact set $K \subseteq G$ such that for every relatively separated set $\Lambda \subseteq G$ with $\rel(\Lambda) \leq \delta$ we have
\[ \Big( \sum_{\lambda \in \Lambda \setminus K} |f(\lambda)|^2 \Big)^{1/2} < \epsilon .\]
\end{lemma}

\begin{definition}
We say that $\Lambda \subseteq G$ is 
\begin{enumerate}
    \item a set of \emph{sampling} for $\Hi$ if there exist $0 < A \leq B < \infty$ such that
\[ A\| f \|^2 \leq \sum_{\lambda \in \Lambda}|f(\lambda)|^2 \leq B \| f \|^2, \quad f \in \Hi .\]
    \item It is a set of \emph{interpolation} for $\Hi$ if for every $(c_\lambda)_{\lambda \in \Lambda} \in \ell^2(\Lambda)$ there exists $f \in \Hi$ such that $f(\lambda) = c_\lambda$ for all $\lambda$. 
\end{enumerate}
\end{definition}

Note that $\Lambda$ is a set of sampling for $\Hi$ by definition if and only if $(k_\lambda)_{\lambda \in \Lambda}$ is a frame for $\Hi$. Analogously \Cref{prop:riesz_limits} shows that $\Lambda$ is a set of interpolation for $\Lambda$ if and only if $(k_\lambda)_{\lambda \in \Lambda}$ is a Riesz sequence for $\Hi$.

\begin{proposition}\label{prop:bessel_relsep}
$( k_\lambda)_{\lambda \in \Lambda}$ is a Bessel sequence if and only if $\Lambda$ is relatively separated. In particular, any set of sampling or interpolation for $\Hi$ is relatively separated.
\end{proposition}

\begin{proof}
The proof of the forward implication can be easily adapted from \cite[Proposition 3.1]{ChDeHe99}. Conversely, suppose that $\Lambda$ is relatively separated. Let $f \in \Hi$. By \Cref{ass:rkhs} (c) $f \in W^2(G)$, so \Cref{cor:wiener_summation} implies in particular that $\sum_{\lambda \in \Lambda} |f(\lambda)|^2 < \infty$. By the closed graph theorem, the linear operator $\Hi \to \ell^2(\Lambda)$, $f \mapsto (f(\lambda))_{\lambda \in \Lambda}$ is bounded, hence $(k_\lambda)_{\lambda \in \Lambda}$ is a Bessel sequence.
\end{proof}

\begin{proposition}\label{prop:riesz_limits}
If $\Lambda$ is a set of interpolation for $\Hi$, then $\Lambda$ is separated and $(k_x)_{x \in P}$ is a Riesz sequence with uniform Riesz bounds $0 < A \leq B < \infty$ for all $P \in \phull{\Lambda}$.
\end{proposition}

\begin{proof}
Note that $\Lambda$ is a set of interpolation if and only if for every $c \in \ell^2(\Lambda)$ there exists $f \in \Hi$ such that $\langle f, k_\lambda \rangle = c_\lambda$ for $\lambda \in \Lambda$. This property is well-known to be equivalent to the existence of a lower Riesz bound $A>0$ for $(k_\lambda)_{\lambda \in \Lambda}$, i.e.,\ $A \| c \|_2^2 \leq \| \sum_\lambda c_\lambda k_\lambda \|$ for all $c \in \ell^2(\Lambda)$, see \cite[p.\ 154--155]{Yo80}. Using the strong continuity of $\pi$ one argues similarly as in \cite[Lemma 2.2 (b)]{GrOrRo15} that $\Lambda$ must be separated. But then by \Cref{prop:bessel_relsep} $(k_\lambda)_{\lambda \in \Lambda}$ must be a Bessel sequence, which is equivalent to the existence of an upper Riesz bound $B > 0$, see e.g.\ \cite[Theorem 3.2.3]{Ch03}. Thus $(k_\lambda)_{\lambda \in \Lambda}$ is a Riesz sequence with bounds $0 < A \leq B < \infty$.

The proof that the Riesz bounds are uniform for every $P \in \phull{\Lambda}$ is a straightforward adaption of Beurling's weak limit techniques, see \cite{Be89}.
\end{proof}

Many variations of the following result exist in the literature, see e.g.\ \cite{Be89,RaSt95,He08}. Since our arguments rely subtly on the Wiener amalgam space $W^2(G)$ and we only assume our point sets to be relatively separated rather than separated as in \cite{RaSt95,He08}, we give a detailed proof.

\begin{theorem}\label{thm:frame_limits}
If $(k_\lambda)_{\lambda \in \Lambda}$ is a frame for $\Hi$ with bounds $0 < A \leq B < \infty$, then $\Lambda$ is relatively separated and relatively dense, and $(k_x)_{x \in P}$ is a frame with frame bounds $0 < A/\rel(\Lambda) \leq B < \infty$ for every $P \in \chull{\Lambda}$.
\end{theorem}

\begin{proof}
Let $(k_\lambda)_{\lambda \in \Lambda}$ be a frame with bounds $0 < A \leq B < \infty$. Then it is in particular a Bessel sequence, hence $\Lambda$ must be $\ell$-relatively $U$-separated for some $\ell \in \N$ and some open neighborhood $U$ of the identity by \Cref{prop:bessel_relsep}. Let $\epsilon > 0$ and let $f \in \Hi$ be nonzero. By \Cref{ass:rkhs} (iii) $f \in W^2(G)$, so by \Cref{cor:wiener_summation} we can find a compact set $K \subseteq G$ such that
\begin{equation}
    \sum_{x \in P' \setminus K}|f(x)|^2 < \frac{\epsilon}{4} \label{eq:estimate_frame_chabauty}
\end{equation}
for all relatively separated sets $P'$ with $\rel_U(P') \leq \ell$. If $\Lambda$ was not relatively dense, then there would exist $y \in G$ such that $y^{-1}\Lambda \cap K = \emptyset$. But then by the frame property and the above
\[ A \| f \|^2 = A\| \Ell_\sigma(y)f \|^2 \leq \sum_{\lambda \in \Lambda}| \langle f, k_{y^{-1}\lambda} \rangle|^2 = \sum_{x \in y^{-1}\Lambda} |\langle f, k_x\rangle|^2 = \sum_{x \in y^{-1}\Lambda \setminus K} |\langle f, k_x \rangle |^2 < \frac{\epsilon}{4} . \]
For $\epsilon$ small enough, this leads to a contradiction, hence $\Lambda$ must be relatively dense.

Let now $P \in \chull{\Lambda}$, say $P = \lim_n P_n$ where each $P_n$ is a left translate of $\Lambda$. Then \eqref{eq:estimate_frame_chabauty} gives
\begin{equation}
    \Big| \sum_{x \in P \setminus K}|f(x)|^2 - \sum_{x \in P_n \setminus K}|f(x)|^2 \Big| < \frac{\epsilon}{2} \label{eq:weak_limit_1}
\end{equation}
for all $n \in \N$.

Enlarging $K$, the estimate \eqref{eq:weak_limit_1} still holds, so that we may without loss of generality assume that $P \cap \partial K = \emptyset$. Indeed, if $b^1, \dotsc, b^r$ is an enumeration of $P \cap \partial K$ and $L \subseteq G$ is a compact neighborhood of the identity such that $P \cap b^iL = \{b^i\}$ for all $i$, then $K' = K \cup b^1L \cup \dotsm \cup b^rL$ satisfies $P \cap \partial K' = \emptyset$ and $K' \supseteq K$.

Let $P \cap K = \{ x^1, \ldots, x^k \}$. Appealing to \Cref{prop:rel_sep_conv} we can find $N \in \N$ such that when $n \geq N$, $P_n \cap K$ is a disjoint union of sets $M_n^1, \ldots, M_n^k$ with $1 \leq |M_n^j|\leq \ell$ and $M_n^j \to \{ x^j \}$ as $n \to \infty$. Passing to a subsequence, we may assume that there is $c^j \in \N$ such that $|M_n^j| = c^j$ for all $n$.  Together with strong continuity of translation, these properties imply that
\[ \lim_{n \to \infty} \sum_{x \in M_n^j}|f(x)|^2 = c^j  |f(x_j)|^2, \quad 1 \leq j \leq k . \]
Define a function $m \colon P \to \N$ by $m(x) = |M_n^j|$ if $x = x^j$ and $m(x) = 1$ otherwise. Using the above, there exists an $N' \geq N$ such that
\begin{equation}
    \Big| \sum_{x \in P \cap K} m(x)| f(x)|^2 - \sum_{x \in P_n \cap K} | f(x)|^2 \Big| < \frac{\epsilon}{2}  \label{eq:weak_limit_2}
\end{equation}
when $n \geq N'$.

Combining \eqref{eq:weak_limit_1} and \eqref{eq:weak_limit_2}, we arrive at
\begin{align*}
    \Big| \sum_{x \in P } m(x)| f(x)|^2 - \sum_{x \in P_n } | f(x)|^2 \Big| &< \frac{\epsilon}{2} + \frac{\epsilon}{2} = \epsilon , \quad \text{ for all } n \geq N'.
\end{align*}
Using that $(k_x)_{x \in P_n}$ is a frame with bounds $0 < A \leq B < \infty$ for every $n \in \N$, we obtain
\[ A \| f \|^2 - \epsilon \leq \sum_{x \in P} m(x)| f(x)|^2 \leq B \| f \|^2 + \epsilon . \]
Since $\epsilon > 0$ was arbitrary, this gives
\[ \frac{A}{\ell} \| f \|^2 \leq \sum_{x \in P} \frac{m(x)}{\ell}|f(x)|^2 \leq \sum_{x \in P}|f(x)|^2 \leq \sum_{x \in P}m(x)|f(x)|^2 \leq B \| f \|^2 . \]
Hence $(k_x)_{x \in P}$ is a frame with bounds $0 < A/\ell \leq B < \infty$.
\end{proof}

\subsection{Necessary density conditions}\label{subsec:necessary-density}

As in the previous subsection we fix a reproducing kernel Hilbert space on $G$ satisfying \Cref{ass:rkhs}. We will now associate to a point set $\Lambda \subseteq G$ a groupoid $\mathcal{G}(\Lambda)$ and a groupoid representation $\pi_\Lambda$ to which we will apply the results from \Cref{subsec:frame_vectors} to deduce \Cref{thm:necessary_conditions}.

\begin{definition}\label{def:groupoid}
The \emph{groupoid} of $\Lambda$, denoted by $\deloid{\Lambda}$, is the restriction of the transformation groupoid $G \ltimes \phull{\Lambda}$ to the transversal $\dhull{\Lambda}$ of $\Lambda$.
\[ \deloid{\Lambda} = \{ (x,P) \in G \times \phull{\Lambda} : P, xP \in \dhull{\Lambda} \} = \{ (x,P) \in G \times \dhull{\Lambda} : x^{-1} \in P \} , \]
with source and range maps given by $s(x,P) = P$ and $r(x,P) = xP$ for $(x,P) \in \deloid{\Lambda}$, multiplication given by $(x,yP)(y,P) = (xy,P)$ and inversion given by $(x,P)^{-1} = (x^{-1},xP)$. We give this groupoid the Borel structure coming from the product topology on the cartesian product $G \times \phull{\Lambda}$.
\end{definition}

For Delone sets in $G = \R^d$, this groupoid has been studied in e.g.\ \cite{BeHeZa00, BoMe19,BoMe21}, often under the name of the ``transversal groupoid'' of $\Lambda$, although with slightly different conventions in the definition than the above definition.

The 2-cocycle $\sigma$ on $G$ gives rise to a measurable 2-cocycle $\sigma_\Lambda$ on $\deloid{\Lambda}$ in the sense of \Cref{subsec:groupoids}. Namely, for any $((x,P), (y,Q)) \in \deloid{\Lambda}^{(2)}$, set
\begin{equation}
    \sigma_{\Lambda}((x,P),(y,Q)) = \sigma(x,y).
\end{equation}
The 2-cocycle identities for $\sigma_{\Lambda}$ follow from the 2-cocycle identites for $\sigma$.

We will now construct a $\sigma_{\Lambda}$-projective groupoid representation $\lregpoint{\sigma}{\Hi}$ of $\deloid{\Lambda}$. As our measurable field of Hilbert spaces $(\Hi_Q)_{Q \in \dhull{\Lambda}}$ over the unit space $\dhull{\Lambda}$ of $\deloid{\Lambda}$, we simply take $\Hi_Q = \Hi$ for every $Q \in \dhull{\Lambda}$. Hence $\prod_{Q \in \dhull{\Lambda}} \Hi_Q$ can be identified with the set of functions $\dhull{\Lambda} \to \Hi$, and we declare the measurable sections to be those functions $f \colon \dhull{\Lambda} \to \Hi$ for which $Q \mapsto \langle f(Q), g \rangle$ is measurable for every $g \in \Hi$.

We define $\lregpoint{\sigma}{\Hi}$ as follows: Given $(x,P) \in \deloid{\Lambda}$, we define $\lregpoint{\sigma}{\Hi}(x,P) \colon \Hi_P \to \Hi_{xP}$ to be 
\[ \lregpoint{\sigma}{\Hi}(x,P) = \lreg{\sigma}{x} , \]
that is, $\lregpoint{\sigma}{\Hi}(x,P)$ is just the $\sigma$-twisted translation operator $\lreg{\sigma}(x)$ on $\Hi = \Hi_P = \Hi_{xP}$. The fact that this defines a $\sigma_{\Lambda}$-projective unitary representation of $\deloid{\Lambda}$ on $(\Hi_Q)_{Q \in \dhull{\Lambda}}$ follows from the fact that the $\sigma$-twisted left regular representation of $G$ is $\sigma$-projective.

The following proposition characterizes frame and Riesz vectors for $\lregpoint{\sigma}{\Hi}$ in the sense of \Cref{subsec:frame_vectors}.

\begin{proposition}\label{prop:frame_vector_char}
Let $\Lambda$ be a relatively separated subset of $G$, let $\mu$ be a non-zero finite invariant measure on $\phull{\Lambda}$ and let $g = (g_Q)_{Q \in \dhull{\Lambda}}$ be a measurable section of $(\Hi_Q)_Q$. Then the following statements are equivalent:
\begin{enumerate}
    \item $g$ is a frame vector for $\lregpoint{\sigma}{\Hi}$.
    \item The families $(\lreg{\sigma}{x}g_{x^{-1}Q})_{x \in Q}$ are frames for each $Q \in \dhull{\Lambda}$, with frame bounds uniform in $Q$.
    \item The families $(\lreg{\sigma}{x}g_{x^{-1}P})_{x \in P}$ are frames for each $P \in \phull{\Lambda}$, with frame bounds uniform in $P$.
\end{enumerate}
Analogous equivalences hold with frame vector replaced by Riesz vector, Parseval frame vector or orthonormal vector.
\end{proposition}

\begin{proof}
By definition, $g$ is a frame vector for $\lregpoint{\sigma}{\Hi}$ if each of the families
\[ \lregpoint{\sigma}{\Hi}(\deloid{\Lambda}^{Q})g = (\lregpoint{\sigma}{\Hi}(x,x^{-1}Q) g_{x^{-1}Q})_{x \in Q} = (\lreg{\sigma}{x}g_{x^{-1}Q})_{x \in Q} \]
is a frame for $\Hip$ with frame bounds uniform in $Q$. This shows that $(i)$ and $(ii)$ are equivalent by definition.

To show that $(ii)$ implies $(iii)$, assume that $g$ is a frame vector for $\lregpoint{\sigma}{\Hi}$. Let $P \in \phull{\Lambda}$ and pick some $y \in P$. Then $y^{-1}P \in \dhull{\Lambda}$ so that
\[ (\lreg{\sigma}{x}g_{x^{-1}y^{-1}P})_{x \in y^{-1}P} = \lreg{\sigma}{y}^* (\lreg{\sigma}{x}g_{x^{-1}P})_{x \in P} \]
is a frame. Since $(\lreg{\sigma}{x}g_{x^{-1}P})_{x \in P}$ is an image of this frame under the unitary operator $\lreg{\sigma}{y}^*$, it follows that the latter is also a frame with the same frame bounds. Finally, the implication from $(iii)$ to $(ii)$ is trivial since $\dhull{\Lambda} \subseteq \phull{\Lambda}$.
\end{proof}

\begin{theorem}\label{thm:frame_bounds_measure}
Let $\Lambda \subseteq G$ be a relatively separated set and let $\mu$ be a nonzero finite invariant measure on $\phull{\Lambda}$. If $g = (g_Q)_{Q \in \dhull{\Lambda}}$ is a frame vector for $\lregpoint{\sigma}{\Hi}$ with frame bounds $0 < A \leq B < \infty$, then
\[ A \leq \frac{ \mu_0(\dhull{\Lambda})^{-1} \int_{\dhull{\Lambda}} \| g_Q \|^2 \dif{\mu_0(Q)} }{\| k_e \|^2 \covol_\mu(\Lambda) }  \leq B .\]
If $g$ is merely a Bessel vector for $\lregpoint{\sigma}{\Hi}$ with Bessel bound $B$, then the upper bound above holds.
\end{theorem}

\begin{proof}
Let $g$ be a frame vector for $\lregpoint{\sigma}{\Hi}$ with frame bounds $0 < A \leq B < \infty$. By \Cref{prop:frame_vector_char}, the following double inequality holds for all $P \in \phull{\Lambda}$:
\[ A \| k_e \|^2 \leq \sum_{x \in P}|\langle k_e, \lreg{\sigma}{x} g_{x^{-1}P} \rangle |^2 \leq B \| k_e \|^2 .\]
Integrating over $\phull{\Lambda}$ with respect to $\mu$, we obtain
\[ A \mu(\phull{\Lambda}) \| k_e \|^2 \leq \int_{\phull{\Lambda}} \sum_{x \in P} | \langle k_e, \lreg{\sigma}{x} g_{x^{-1}P} \rangle|^2 \dif{\mu(P)} \leq B \mu(\phull{\Lambda}) \| k_e \|^2 .\]
We consider the middle term of the above double inequality. Using the transverse measure formula \eqref{eq:transverse} and the unimodularity of $G$, we obtain
\begin{align*}
    \int_{\phull{\Lambda}} \sum_{x \in P}|\langle k_e, \lreg{\sigma}{x} g_{x^{-1}P} \rangle |^2 \dif{\mu(P)} &= \int_{\dhull{\Lambda}} \int_G | \langle k_e, \lreg{\sigma}{x} g_Q \rangle |^2 \dif{\mu_0(Q)} \\
    &= \int_{\dhull{\Lambda}} \int_G |\langle k_{x^{-1}}, g_Q \rangle|^2 \dif{x} \dif{\mu_0(Q)} \\
    &= \int_{\dhull{\Lambda}} \int_G | g_Q(x^{-1})|^2 \dif{x} \dif{\mu_0(Q)} \\
    &= \int_{\dhull{\Lambda}} \| g_Q \|^2 \dif{\mu_0(Q)} .
\end{align*}
Dividing throughout the double inequality by $\| k_e \|^2 \mu_0(\dhull{\Lambda})$, the conclusion follows.
\end{proof}

Everything is now in place to prove the following theorem, which contains the main result \Cref{thm:necessary_conditions} as the special case $\sigma = 1$.

\begin{theorem}\label{thm:twisted_necessary_conditions}
Let $\Lambda$ be a relatively separated subset of a second-countable, locally compact group $G$ and let $\Hi$ be a reproducing kernel Hilbert space on $G$ satisfying \Cref{ass:rkhs}. Then the following hold:
\begin{enumerate}
    \item If $\Lambda$ is a set of sampling for $\Hi$, then $\| k_e \|^2 \covol_+(\Lambda) \leq 1$.
    \item If $\Lambda$ is a set of interpolation for $\Hi$, then $\| k_e \|^2 \covol_-(\Lambda) \geq 1$.
\end{enumerate}
\end{theorem}

\begin{proof}
First, if there exist no non-zero finite invariant measures on $\phull{\Lambda}$ then $\covol_-(\Lambda) = \infty$ while $\covol_+(\Lambda) = 0$. Since $\| k_e \|^2$ is positive, the the statements of the theorem hold vacuously. We may thus assume that there is a non-zero finite invariant measure $\mu$ on $\phull{\Lambda}$.

$(i)$: Suppose that $\Lambda$ is a set of sampling for $\Hi$. By \Cref{thm:frame_limits} $\Lambda$ is relatively separated and relatively dense, and $(k_x)_{x \in P}$ is a frame for every $P \in \chull{\Lambda}$ with frame bounds uniform in $P$. By \Cref{prop:frame_vector_char} the constant section $(k_e)_{P \in \dhull{\Lambda}}$ is a frame vector for $\lregpoint{\sigma}{\Hi}$. By \Cref{cor:frame_iff_parseval} there exists a Parseval frame vector $g = (g_Q)_{Q \in \dhull{\Lambda}}$ for $\lregpoint{\sigma}{\Hi}$. It then follows from \Cref{thm:frame_bounds_measure} that $\mu_0(\dhull{\Lambda})^{-1} \int_{\dhull{\Lambda}} \| g_Q \|^2 \dif{\mu_0(Q)} = \| k_e \|^2 \covol_\mu(\Lambda)$. Since each $Q \in \dhull{\Lambda}$ contains $e$, the vector $g_Q = \lreg{\sigma}{e} g_{e^{-1}Q}$ is an element of the Parseval frame $(\lreg{\sigma}{x}g_{x^{-1}Q})_{x \in Q}$. Since vectors in Parseval frames have norm at most $1$, it follows that
\[ \| k_e \|^2 \covol_\mu(\Lambda) = \frac{1}{\mu_0(\dhull{\Lambda})} \int_{\dhull{\Lambda}} \| g_Q \|^2 \dif{\mu_0(Q)} \leq 1 . \]
As $\mu$ was arbitrary, we conclude that $\| k_e \|^2 \covol_+(\Lambda) \leq 1$.

$(ii)$: Suppose that $\Lambda$ is a set of interpolation for $\Hi$. By \Cref{prop:riesz_limits} $\Lambda$ is separated and $(k_x)_{x \in P}$ is a Riesz sequence for every $P \in \phull{\Lambda}$ with Riesz bounds uniform in $P$. By \Cref{prop:frame_vector_char} the constant section $(k_e)_{P \in \dhull{\Lambda}}$ is a Riesz vector for $\lregpoint{\sigma}{\Hi}$. By \Cref{cor:frame_iff_parseval} there exists an orthonormal vector $g$ for $\lregpoint{\sigma}{\Hi}$, in particular a Bessel vector for $\lregpoint{\sigma}{\Hi}$ with Bessel bound $B=1$. It then follows from \Cref{thm:frame_bounds_measure} that $\mu_0(\dhull{\Lambda})^{-1} \int_{\dhull{\Lambda}} \| g_Q \|^2 \dif{\mu_0(Q)} \leq \| k_e \|^2 \covol_\mu(\Lambda)$. Similarly as in $(ii)$, each vector $g_Q$ is an element of an orthonormal sequence, hence $\| k_e \|^2\covol_\mu(\Lambda) \geq 1$. As $\mu$ was arbitrary, we conclude that $\| k_e \|^2 \covol_-(\Lambda) \geq 1$.
\end{proof}

\begin{remark}\label{rmk:interpolating}
\Cref{ass:rkhs} (c) is essential to prove both parts (i) and (ii) of \Cref{thm:twisted_necessary_conditions}. However, the following variation of (ii) is valid without \Cref{ass:rkhs} (c):
\begin{enumerate}[label=(\roman*')]
    \addtocounter{enumi}{1} \item  If $(k_\lambda)_{\lambda \in \Lambda}$ is a Riesz sequence for $\Hi$, then $\| k_e \|^2 \covol_-(\Lambda) \geq 1$.
\end{enumerate}
Indeed, the proof of \Cref{thm:twisted_necessary_conditions} (ii) hinges on \Cref{prop:riesz_limits} which uses \Cref{ass:rkhs} (c) only at one point, namely to show that $\Lambda$ being interpolating implies that $(k_\lambda)_{\lambda \in \Lambda}$ is a Riesz sequence. Hence if $(k_\lambda)_{\lambda \in \Lambda}$ is already assumed to be a Riesz sequence like above, then \Cref{ass:rkhs} (c) is not needed. Also, if $G$ is assumed to be an IN group, then \Cref{ass:rkhs} (c) is not needed for \Cref{thm:twisted_necessary_conditions} (i). In this case the proof given in \cite[Lemma 5.19]{He08} which does not assume any localization can replace \Cref{thm:frame_limits}.
\end{remark}

\begin{corollary}\label{cor:necessary_conditions_amenable}
Let $\Lambda$ be a subset of an amenable, unimodular lcsc group $G$ and denote by $D^{-}(\Lambda)$ and $D^{+}(\Lambda)$ respectively the lower and upper Beurling density of $\Lambda$ with respect to a strong Følner sequence. Let $\Hi$ be a reproducing kernel Hilbert space on $G$ as in \Cref{thm:necessary_conditions}. Then the following hold:
\begin{enumerate}
    \item If $\Lambda$ is a set of sampling for $\Hi$, then $D^-(\Lambda) \geq \| k_e \|^2$.
    \item If $\Lambda$ is a set of interpolation for $\Hi$, then $D^+(\Lambda) \leq \| k_e \|^2$.
\end{enumerate}
\end{corollary}

\begin{proof}[Proof of \Cref{cor:necessary_conditions_amenable}]
$(i)$ If $\Lambda$ is a is a set of sampling for $\Hi$, then $\Lambda$ is relatively separated and relatively dense by \Cref{thm:frame_limits} and $\| k_e \|^2 \covol_+(\Lambda) \leq 1$ by \Cref{thm:necessary_conditions}. Since $D^{--}(\Lambda) = 1/\covol_+(\Lambda)$ by \Cref{thm:density_covolume} we conclude that $D^-(\Lambda) \geq D^{--}(\Lambda) \geq d_\pi$.

$(ii)$ If $\Lambda$ is a set of interpolation for $\Hi$ then $\Lambda$ is separated by \Cref{prop:riesz_limits} and the inequality $\| k_e \|^2 \covol_-(\Lambda) \geq 1$ holds by \Cref{thm:necessary_conditions}. Since $D^{++}(\Lambda) = 1/\covol_-(\Lambda)$ by \Cref{thm:density_covolume} and $D^{+}(\Lambda) = D^{++}(\Lambda)$ by \Cref{prop:beurling_chull} we conclude that $D^{+}(\Lambda) \leq \| k_e \|^2$.
\end{proof}

\section{Examples}\label{sec:examples}

\subsection{Coherent systems arising from projective discrete series representations}\label{subsec:coherent}

Let $G$ be a unimodular lcsc group with 2-cocycle $\sigma$. A \emph{$\sigma$-projective unitary representation} of $G$ on a Hilbert space $\Hip$ is a map $\pi \colon G \to \mathcal{U}(\Hip)$, where $\mathcal{U}(\Hip)$ denotes the unitary operators on $\Hip$, such that
\[ \pi(x)\pi(y) = \sigma(x,y)\pi(xy) \;\;\; \text{for all $x,y \in G$,} \]
and such that $\pi$ is \emph{strongly continuous}, that is, $x \mapsto \pi(x)\xi$ is continuous for every $\xi \in \Hip$.

A $\sigma$-projective, irreducible, unitary representation $\pi$ of $G$ is called a \emph{projective discrete series representation} if there exists a non-zero vector $\xi \in \Hip$ such that
\[ \int_G | \langle \xi, \pi(x) \xi \rangle|^2 \dif{x} < \infty .\]
If $\pi$ is a projective discrete series representation, then in fact the above integral is finite for all $\xi \in \Hip$, and moreover one has the orthogonality relations
\begin{equation}
    \int_G \langle \xi, \pi(x) \eta \rangle \overline{ \langle \xi', \pi(x) \eta' \rangle } \dif{x} = d_{\pi}^{-1} \langle \xi, \xi' \rangle \overline{ \langle \eta, \eta' \rangle } \;\;\; \text{for all $\xi,\xi',\eta,\eta' \in \Hip$,} \label{eq:orthogonality}
\end{equation}
where the number $d_{\pi}$ is called the \emph{formal dimension of $\pi$}. It depends on the chosen Haar measure on $G$. If the Haar measure is scaled by a factor of $c > 0$, then the formal dimension is scaled by $c^{-1}$.

Given a discrete set $\Lambda \subseteq G$ and $\eta \in \Hip$, a sequence in $\Hip$ of the form
\[ \pi(\Lambda)\eta = (\pi(\lambda)\eta)_{\lambda \in \Lambda} \]
is often called a \emph{coherent system}.

Given $\xi,\eta \in \Hip$, we will denote by $C_\eta \xi \colon G \to \C$ the associated matrix coefficient given by
\[ C_\eta \xi (x) = \langle \xi, \pi(x) \eta \rangle, \quad x \in G .\]
As shown in e.g.\ \cite[Section 5.3]{FuGrHa17} the space $\Hi = C_\eta \Hip$ is a reproducing kernel Hilbert space satisfying \Cref{ass:rkhs} provided that $\eta$ belongs to the space
\[ \mathcal{B}_\pi = \{ \eta \in \Hip : \text{$C_{\eta}\xi \in W^2(G)$ for all $\xi \in \Hip$} \} .\]
The elements of this space were termed \emph{admissible analyzing vectors} in \cite{Gr08}. A set $\Lambda \subseteq G$ is sampling (resp.\ interpolating) for $\Hi$ if and only if $\pi(\Lambda)\eta$ is a frame (resp.\ Riesz sequence) for $\Hip$. Thus \Cref{thm:twisted_necessary_conditions} and \Cref{rmk:interpolating} apply to give the following theorem:

\begin{theorem}\label{thm:coherent_systems}
Let $(\pi,\Hip)$ be a projective discrete series representation of a unimodular lcsc group $G$ and let $\Lambda \subseteq G$.
\begin{enumerate}
    \item If $\pi(\Lambda)\eta$ is a frame for $\Hip$ where $\eta \in \mathcal{B}_\pi$, then $d_\pi \covol_+(\Lambda) \leq 1$.
    \item If $\pi(\Lambda)\eta$ is a Riesz sequence for $\Hip$ where $\eta \in \Hip$, then $d_\pi \covol_-(\Lambda) \geq 1$.
\end{enumerate}
\end{theorem}

\subsection{Sampling and interpolation in Paley--Wiener spaces}\label{subsec:paley}

In this subsection we assume that $G$ is an abelian lcsc group and denote by $\widehat{G}$ its Pontryagin dual. We define the Fourier transform according to the convention
\[ \widehat{f}(\omega) = \int_G f(x) \overline{\omega(x)} \dif{x}, \quad f \in L^1(G) ,\]
and to a choice of Haar measure on $G$ we implicitly integrate with respect to the corresponding dual Haar measure on $\widehat{G}$ so that the Plancherel formula holds:
\[ \int_G |f(x)|^2 \dif{x} = \int_{\widehat{G}} |\widehat{f}(\omega)|^2 \dif{\omega}, \quad f \in L^2(G) . \]
Given a relatively compact set $K \subseteq \widehat{G}$, the corresponding \emph{Paley--Wiener} space on $G$ is defined as
\[ PW_K(G) = \{ f \in L^2(G) : \text{$\widehat{f}(\omega) = 0$ for all $\omega \notin K$} \} .\]
It is well-known that the Paley--Wiener space is a $G$-invariant reproducing kernel Hilbert subspace of $L^2(G)$ with kernel determined by $\widehat{k_e} = \mathbbm{1}_K$. In particular $\| k_e \|^2 = \widehat{m}(K)$. An argument that \Cref{ass:rkhs} (c) is satisfied can be found in \cite[Lemma 1]{GrRa96}. Since $G$ is abelian, it is in particular amenable, so \Cref{cor:necessary_conditions_amenable} applies and we obtain the following:

\begin{theorem}\label{thm:paley_wiener}
Let $G$ be a abelian lcsc group and let $K$ be a relatively compact subset of $\widehat{G}$.
\begin{enumerate}
    \item If $\Lambda$ is a set of sampling for $PW_K(G)$, then $D^-(\Lambda) \geq \widehat{m}(K)$.
    \item If $\Lambda$ is a set of interpolation for $PW_K(G)$, then $D^+(\Lambda) \leq \widehat{m}(K)$.
\end{enumerate}
\end{theorem}

For $G = \R^d$, the above theorem is a classical result of H.\ Landau \cite{La67}. Its generalization to locally compact abelian groups was proved in \cite{GrKuSe08}, although with a slightly different formulation than with Beurling densities, using instead structure of compactly generated abelian groups to compare with a certain lattice. A version of the necessary conditions in terms of Beurling densities along strong Følner sequences was given in \cite{RiSc20}, where it was shown to be equivalent to the main result of \cite{GrKuSe08}.

\appendix

\section{More on groupoids of point sets}\label{sec:point-set-groupoid}

In this appendix we consider the groupoid of a point set as a topological groupoid and characterize when it is étale.

Recall that a \emph{locally compact groupoid} is a groupoid $\mathcal{G}$ equipped with locally compact topology for which the maps $s$, $r$, multiplication and inversion are continuous with respect to the relative topologies of $\G^{(0)}$ in $\G$ and $\G^{(2)}$ in the product $\G \times \G$. We also assume that $\G$ is Hausdorff, which implies that $\G^{(0)}$ is closed in $\G$. A locally compact groupoid $\G$ is called \emph{étale} if the source and range maps are local homeomorphisms. For étale groupoids, the unit space $\G^{(0)}$ is open in $\G$ and the fibers $\G_{\omega}$ and $\G^{\omega}$ are discrete subsets of $\G$ for each $\omega \in \G^{(0)}$. An open bisection is a subset $\mathcal{U} \subseteq \G$ such that the restrictions of $s$ and $r$ to $\mathcal{U}$ are injective. An groupoid is étale if and only if its range and source maps are open and it admits a basis consisting of open bisections.

An \emph{abstract transversal} for a locally compact groupoid $\G$ is a closed subset $\Omega_0 \subseteq \Omega$ such that the following conditions hold:
\begin{enumerate}
    \item $\Omega_0$ meets every orbit of $\Omega$ as a right $\G$-space, i.e.,\ for every $\omega \in \Omega$ the intersection $\Omega_0 \cap \{ r(\alpha) : \alpha \in \G_\omega \}$ is nonempty.
    \item With $Z \coloneqq \{ \alpha \in \G : r(\alpha) \in \Omega_0 \}$,
    the restrictions $s \colon Z \to \Omega$ and $r \colon Z \to \Omega_0$ are open maps for the relative topologies on $Z$ and $\Omega_0$.
\end{enumerate}

The reason to consider transversals is that the space $Z$ establishes a groupoid equivalence between $\G$ and the restriction groupoid $\G|_{\Omega_0}$, cf.\ \cite[Example 2.7]{MuReWi87}.

\begin{proposition}\label{prop:transversal_groupoid}
Let $\Lambda$ be a relatively separated subset of a locally compact group $G$. The transversal $\dhull{\Lambda}$ is an abstract transversal for the transformation groupoid $G \ltimes \phull{\Lambda}$.
\end{proposition}

\begin{proof}
First we prove that every orbit of $\chull{\Lambda}$ as a right $G \ltimes \chull{\Lambda}$-space meets $\dhull{\Lambda}$. The orbit of $P \in \phull{\Lambda}$ is given by the set $\{ xP : x \in G \}$. Picking any $x \in P^{-1}$, one has $e \in xP$, hence $xP \in \dhull{\Lambda}$.

Consider the space $Z(\Lambda) = \{ (x,P) \in G \times \phull{\Lambda} : x^{-1} \in P \}$. We will prove that $s \colon Z(\Lambda) \to \phull{\Lambda}$ and $r \colon Z(\Lambda) \to \dhull{\Lambda}$ are open maps. It suffices to check that the image of $(V \times W) \cap Z(\Lambda)$ under $s$ is open in $\phull{\Lambda}$ and the corresponding image under $r$ is open in $\dhull{\Lambda}$ for every open $V$ in $G$ and open $W$ in $\chull{\Lambda}$. For $s$, note that
\[ s((V \times W) \cap Z(\Lambda)) = \{ P \in W : \text{$x^{-1} \in P$ for some $x \in V$} \} = W \cap \mathcal{O}^{V^{-1}} , \]
which is open in $\phull{\Lambda}$. Further,
\[ r((V \times W) \cap Z(\Lambda)) = \{ xP : x \in V, P \in W \} \cap \dhull{\Lambda} = \Big( \bigcup_{x \in V}xW \Big) \cap \dhull{\Lambda} \]
is open in $\dhull{\Lambda}$ by definition of the subspace topology and the continuity of the action of $G$ on $\chull{\Lambda}$. This finishes the proof.
\end{proof}

We now will show that the groupoid $\deloid{\Lambda}$ as in \Cref{def:groupoid} is {\'e}tale if and only if the point set $\Lambda$ is separated and provide an explicit basis of open bisections in this case. For this we need a lemma.

\begin{lemma}\label{lem:unit_space}
The space $\dhull{\Lambda}$ is open in $\deloid{\Lambda}$ if and only if $\Lambda$ is separated.
\end{lemma}

\begin{proof}
Assume that $\dhull{\Lambda}$ is not open. Then there exist $P \in \dhull{\Lambda}$ and a sequence $(x_n,P_n) \in \deloid{\Lambda}$ such that $(x_n,P_n) \to (e,P)$ and $x_n \neq e$. For any identity neighborhood $U \subseteq G$ we can find $n(U) \in \N$ such that $x_{n(U)}^{-1} \in U$. Hence $|P_{n(U)} \cap U| \geq 2$. Since $\{ x\Lambda : x \in G \}$ is dense in $\phull{\Lambda}$, we may by \Cref{prop:U_discrete_closed} (i) find $y_{n(U)} \in G$ such that $|y_{n(U)} \Lambda \cap U| \geq 2$. But then $|\Lambda \cap y_{n(U)}^{-1}U| \geq 2$, so $\Lambda$ is not $U$-separated. Since $U$ was arbitrary, this shows that $\Lambda$ is not separated.

Conversely, assume that $\Lambda$ is not separated. Then for every open identity neighborhood $U \subseteq G$ there is $y \in G$ such that $|yU \cap \Lambda| \geq 2$. We may assume that $y \in \Lambda$: Indeed, letting $V$ be an open identity neighborhood such that $V^2 \subseteq U$, we can by assumption find $y' \in G$ and two distinct $\lambda,\lambda' \in \Lambda \cap y'V$. Then $y'^{-1}\lambda,y'^{-1}\lambda' \in V$, so $(y'^{-1}\lambda)^{-1}(y'^{-1}\lambda') = \lambda^{-1}\lambda' \in V^2 \subseteq U$. Hence both $\lambda \in \Lambda \cap \lambda U$ and $\lambda' = \lambda(\lambda^{-1}\lambda') \in \Lambda \cap \lambda U$, so $|\Lambda \cap \lambda U| \geq 2$.

Having for every identity neighborhood $U \subseteq G$ found $\lambda \in \Lambda$ such that $|\Lambda \cap \lambda U| = |\lambda^{-1}\Lambda \cap U| \geq 2$, we may pick translates $(P_n)_n$ of $\Lambda$ containing the identity element and elements $e \neq x_n \in P_n$ such that $x_n \to e$.  By compactness, we may assume that $P_n \to P \in \chull{\Lambda}$ and it follows that $P \in \dhull{\Lambda}$.  The fact that $(x_n^{-1}, P_n) \to (e, P)$ in $\deloid{\Lambda}$ shows that the space of units of $\deloid{\Lambda}$ is not open.
\end{proof}

\begin{proposition}
The groupoid $\deloid{\Lambda}$ is étale if and only if $\Lambda$ is separated. Moreover, let $\Lambda$ be $U$-separated with $U$ a symmetric, open neighborhood of the identity. Then as $x$ ranges over $G$, as $V \subseteq G$ ranges over all symmetric, open neighborhoods of the identity with $V^2 \subseteq U$ and as $W$ ranges over open sets in $\dhull{\Lambda}$, the sets
\[ \mathcal{U}_{x,V,W} = ((xV \cap Vx) \times W) \cap \deloid{\Lambda} \]
form a basis for the topology of $\deloid{\Lambda}$ consisting of open bisections of $\deloid{\Lambda}$.
\end{proposition}

\begin{proof}
If $\Lambda$ is not separated then $\dhull{\Lambda}$ is not open in $\deloid{\Lambda}$ by \Cref{lem:unit_space}. In particular $\deloid{\Lambda}$ is not étale.

Conversely, suppose $\Lambda$ is $U$-separated with $U$ a symmetric, open identity neighborhood. For $x$ and $V$ as in the statement of the proposition, set $V_x = xV \cap Vx$. Since $\deloid{\Lambda}$ is equipped with the subspace topology inherited from $G \times \Omega(\Lambda)$ and the sets $V_x \times W$ form a basis for the topology of the latter, it follows that the sets $(V_x \times W) \cap \deloid{\Lambda}$ form a basis for the topology on $\deloid{\Lambda}$.

We now show that $s$ is a local homeomorphism. As already observed in the proof of \Cref{prop:transversal_groupoid} the image of $\mathcal{U}_{x,V,W}$ under $s$ is given by
\begin{align*}
    s(\mathcal{U}_{x,V,W}) = W \cap \mathcal{O}^{V_x^{-1}} .
\end{align*}
Thus $s(\mathcal{U}_{x,V,W})$ is open in $\dhull{\Lambda}$. Since $\dhull{\Lambda}$ is open in $\deloid{\Lambda}$ by \Cref{lem:unit_space}, it follows that $s(\mathcal{U}_{x,V,W})$ is open in $\deloid{\Lambda}$. Furthermore, since $|P \cap V_x^{-1}| \leq |P \cap x^{-1}V| \leq 1$ for every $P \in \phull{\Lambda}$, it follows that $|P \cap V_x^{-1}| = 1$, say $P \cap V_x^{-1} = \{ x_P \}$ for each $P \in W \cap \mathcal{O}^{V_x^{-1}}$. The assignment $W \cap \mathcal{O}^{V_x^{-1}} \to V_x^{-1}$ given by $P \mapsto x_P$ is continuous: Indeed, let $V'$ be any open neighborhood of $x_P$ in $V_x^{-1}$. Then $W \cap \mathcal{O}^{V'}$ is an open neighborhood of $P$. For any $Q \in W \cap \mathcal{O}^{V'}$ we have that $1 \leq |Q \cap V'| \leq |Q \cap V_x^{-1}| \leq 1$, so $Q \cap V' = \{ x_Q \}$. Hence $x_Q \in V'$, which proves continuity. It follows that the map $W \cap \mathcal{O}^{V_x^{-1}} \to \mathcal{U}_{x,V,W}$ given by $P \mapsto (x_P^{-1},P)$ is a continuous inverse to $s|_{\mathcal{U}_{x,V,W,}}$. Hence $s$ is a local homeomorphism, so $\deloid{\Lambda}$ is étale.

In order to show that $\mathcal{U}_{x,V,W}$ is an open bisection, it remains to show that $r$ is injective when restricted to this set. But if $xP = x'P'$ for $(x,P),(x',P') \in \mathcal{U}_{x,V,W}$, we find that $e \in P = x^{-1}x' P'$.  So $e,(x')^{-1}x \in P' \cap (xV)^{-1}(xV) \subseteq P' \cap U$. Since $P'$ is $U$-separated, it follows that $x'^{-1}x = e$, which immediately shows that $P = P'$ holds, too. This finishes the proof.
\end{proof}

\printbibliography

\end{document}